\documentclass[11pt,reqno]{amsart}
\usepackage{amsfonts,amssymb,hyperref,amsthm,enumerate,color,stmaryrd,color}
\usepackage[left=2.6cm,right=2.6cm,top=3cm,bottom=3cm,bindingoffset=0cm]{geometry}
\newtheorem{thm}{Theorem}[section]
\newtheorem{lem}[thm]{Lemma}
\newtheorem{prop}[thm]{Proposition}
\newtheorem{cor}[thm]{Corollary}
\theoremstyle{remark}

\def\A{\mathcal{A}}
\def\B{\mathcal{B}}
\def\cS{\mathcal{S}}

\def\L{\mathcal{L}}
\def\N{\mathcal{N}}
\def\O{\mathcal{O}}
\def\P{\mathbf{P}}
\def\r{_R}

\def\Z{\mathbb{Z}}
\def\trX{X^{\circ}}
\def\trA{\A^{\circ}}
\DeclareMathOperator{\aut}{Aut}
\DeclareMathOperator{\cay}{Cay}
\DeclareMathOperator{\iso}{Iso}
\DeclareMathOperator{\orb}{Orb}
\DeclareMathOperator{\rad}{rad}
\DeclareMathOperator{\rk}{rk}
\DeclareMathOperator{\Span}{Span}
\DeclareMathOperator{\Sup}{Sup}
\DeclareMathOperator{\sym}{Sym}
\DeclareMathOperator{\cyc}{Cyc}

\begin{document}
\title[CI-property of $C_p^2 \times C_n$ and $C_p^2 \times C_q^2$ for digraphs]
{CI-property of $\boldsymbol{C_p^2 \times C_n}$ and
$\boldsymbol{C_p^2 \times C_q^2}$ for digraphs}
\author[I.~Kov\'acs]{Istv\'an~Kov\'acs$^{\, 1,2,3}$}
\address{I.~Kov\'acs 
\newline\indent
UP IAM, University of Primorska, Muzejski trg 2, SI-6000 Koper, Slovenia 
\newline\indent
UP FAMNIT, University of Primorska, Glagol\v jaska ulica 8, SI-6000 Koper, Slovenia}
\email{istvan.kovacs@upr.si}
\author[M.~Muzychuk]{Mikhail~Muzychuk}
\address{M.~Muzychuk
\newline\indent
Department of Mathematics, Ben Gurion University of the Negev, 
84105 Beer Sheva, Israel}
\email{muzychuk@bgu.ac.il}
\author[P.~P.~P\'alfy]{P\'eter~P.~P\'alfy$^{\, 2}$}
\address{P.~P.~P\'alfy
\newline\indent 
Alfr\'ed R\'enyi Institute of Mathematics, 
Re\'altanoda utca 13-15, H-1053 Budapest, Hungary}
\email{ppp@renyi.hu}
\author[G.~Ryabov]{Grigory~Ryabov$^{\, 3, 4}$}
\address{G.~Ryabov
\newline\indent  
Sobolev Institute of Mathematics, Acad. Koptyug avenue 4, 630090, Novosibirsk, Russia 
\newline\indent 
Novosibirsk State University, Pirogova str. 1, 630090, Novosibirsk, Russia
\newline\indent  
Novosibirsk State Technical University, K. Marksa avenue 20, 630073, Novosibirsk, Russia}
\email{gric2ryabov@gmail.com}
\author[G.~Somlai]{G\'abor~Somlai$^{\, 2, 5}$}
\address{G.~Somlai
\newline\indent
E\"otv\"os Lor\'and University, Department of Algebra and Number Theory, 
P\'azm\'any P\'eter S\'et\'any 1/C, H-1117 Budapest, Hungary} 
\email{gabor.somlai@ttk.elte.hu}
\thanks{$^1$~Supported by the Slovenian Research Agency 
(research program P1-0285, research projects N1-0062, J1-9108, J1-1695, J1-2451 and N1-0208).
\newline\indent
$^2$~Supported by the ARRS-NKFIH Slovenian-Hungarian Joint Research Project, 
grant no.~SNN 132625 (in Hungary) and N1-0140 (in Slovenia). 
\newline\indent
$^3$~Supported by the Slovenian Research Agency (bilateral project BI-RU/19-20-032).
\newline\indent
$^4$~Supported by the Mathematical Center in Akademgorodok under the agreement No. 075-15-2019-1613 with the Ministry of Science and Higher Education of the Russian
Federation.
\newline\indent
$^5$~Supported by the J\'anos Bolyai Research Fellowship and by
the New National Excellence Program under the grant number UNKP-20-5-ELTE-231.   
}
\maketitle
\begin{abstract}
We prove that the direct product of two coprime order elementary abelian groups of rank two, as well as the direct product of a cyclic group of prime order and a cyclic
group of square free order are DCI-groups. The latter is a generalization of Muzychuk's result on cyclic groups (J.~Combin.~Theory Ser.~A, 1995).
\end{abstract}

\section{Introduction}\label{sec:intro}
Investigation of the isomorphism problem of Cayley graphs started in 1967 with the following conjecture
of \'Ad\'am~\cite{adam}. He asked whether two circulant graphs on $n$ vertices are isomorphic if and only if they are isomorphic via a multiplication with an integer coprime to $n$.

A generalisation of the question using a different terminology was introduced in  \cite{babaifrankl}. Let $G$ be a finite group and let $S$ be a subset of $G \setminus \{ e \}=G^\#$. The vertices of the \emph{Cayley graph} $\cay(G,S)$ are the elements of $G$ and $g \in G$ is connected to $h \in G$ if and only if $hg^{-1} \in S$. 
A right multiplication by a group element $g \in G$ is an automorphism of an arbitrary Cayley graph and hence $\aut(\cay(G,S))$ contains a regular subgroup isomorphic to $G$.

Any automorphism $\alpha$ of $G$ induces an isomorphism between the two Cayley graphs $\cay(G, S)$ and $\cay(G,S^{\alpha})$. In this case these graphs are called \emph{Cayley isomorphic}. A subset $S$ of the group $G$ is \emph{CI} if
$\cay(G,S) \cong \cay(G,T)$ implies that these graphs are Cayley isomorphic. A group $G$ is called a \emph{DCI-group} if
$S$ is CI for every  $S \subset G^\#$ and it is called a \emph{CI-group} if the same holds for symmetric ($S^{-1}=S$) subsets of  $G^\#$.

The first counterexample for \'Ad\'am's conjecture was given by Elspas and Turner~\cite{ET}, and independently by Djokovi\v{c}~\cite{D}. The complete description of finite cyclic DCI-groups was given by Muzychuk~\cite{M2} in 1997, who proved that a cyclic group $C_n$ is DCI if and only if $n=ab$, where $a \mid 4$ and $b$ is a square free odd number.

The class of CI-groups is closed under taking subgroups. It was proved by Babai and 
Frankl~\cite{babaifrankl} that a finite $p$-group is a DCI-group only if it is either an elementary abelian $p$-group or a quaternion group of order $8$ or a cyclic group of small order. This poses a strong restriction on the structure of DCI-groups. A
collection of the candidates of (D)CI-groups is found in \cite{LLP}. 
Recently, further significant restriction was obtained by Dobson et al.~\cite{DMS}.
Furthermore, it has been proved by Muzychuk~\cite{M0} that for every prime $p$, the elementary abelian $p$-groups of sufficiently large rank are not CI-groups. The current lower bound for the rank of a non-CI elementary abelian $p$-group is 
$2p+3$ \cite{Som}. On the other hand, it was proved by Feng and Kov\'acs~\cite{FK} that $C_p^5$ is CI-group for every prime $p$.

It was conjectured by Kov\'acs and Muzychuk~\cite{KM} that the direct product of 
DCI-groups of coprime orders is always a DCI-group 
(see also \cite[Conjecture~43]{Dob}). They proved that $C_p^2 \times C_q$ is a 
DCI-group for every pair of distinct primes $p$, $q$. As a strengthening of this result it was proved that $C_p^3 \times C_q$ and $C_p^4 \times C_q$ are also DCI-groups, 
see \cite{SM,KR2}. Furthermore, Dobson~\cite{Dob} settled the conjecture for 
abelian groups under strong restrictions on the order of the factors.

It was shown by Babai~\cite{B} that $S \subseteq G^\#$ is CI if and only if all regular subgroups of $\aut(\cay(G,S))$ isomorphic to $G$ are conjugate in the automorphism group. This observation gives us one of the main tools in the study of (D)CI-groups and allows us to use results from group theory. Another basic method in these investigations is the use of Schur rings. It started in the paper of Klin and P\"oschel~\cite{KP}, where it was proved that a cyclic group whose order is a product of two different primes is a DCI-group. The method was further developed in a paper of Hirasaka and Muzychuk~\cite{HM}, where the notion of star product was introduced. 
We refer to the survey paper \cite{MP} for more information on Schur 
rings and their link with combinatorics.

In our paper the techniques developed in \cite{SM} will be combined with a criterion given in \cite{KR} to lead to our results.
\begin{thm}\label{main1}
For any prime $p$ and any square free number $n$ the group $C_p \times C_n$ is a 
DCI-group.
\end{thm}
If $n$ is not divisible by $p$, then $C_p \times C_n \cong C_{pn}$ is a cyclic group of square free order, so this result includes
Muzychuk's theorem \cite{M1} and our methods provide an independent proof for that. If $p$ divides $n$, then $C_p \times C_n \cong C_p^2 \times C_{n/p}$
belongs to a new class of groups for which we establish the CI property.

\begin{thm}\label{main2}
If $p$ and $q$ are different primes, then $C_p^2 \times C_q^2$ is a DCI-group.
\end{thm}
This theorem provides the first example besides elementary abelian $p$-groups of an infinite family of DCI-groups, which are not Burnside groups. The proof of 
Theorem~\ref{main2} uses some techniques from finite geometry.

As a consequence of Theorems~\ref{main1} and \ref{main2}, and results in \cite{HM,M2,Ry,SM}, we have the complete list of abelian 
DCI-groups whose order is a product of four not necessarily distinct primes.

\begin{thm}\label{main3}
The abelian DCI-groups whose order is a product of four not necessarily distinct primes 
are the following groups:
$$
C_p^4,~C_p^3 \times C_q,~C_p^2 \times C_q^2,~C_p^2 \times C_{qr},~C_r^2 \times 
C_4,~C_{4rs},~C_{pqrs},
$$
where $p, q, r, s$ are pairwise distinct primes and $r, s > 2$. 
\end{thm}

\medskip

The paper is organised as follows. The concept of Schur rings is presented in 
Section~\ref{sec:S-r1}.
The next two sections are devoted to preparation for the proof of our two main theorems. The main result in Section~\ref{sec:S-r2} is Lemma~\ref{sylow} that can certainly be applied to other infinite families of abelian groups. Section~\ref{sec:S-r3} collects results on different types of products of CI-S-rings. The proof of Theorem~\ref{main1} is contained in Section~\ref{sec:proof1}. Section~\ref{sec:S-r4} is devoted to the investigation of uniprimitive groups containing a regular subgroup isomorphic to 
$C_p^2 \times C_q^2$ using translation nets. The proof of Theorem~\ref{main2} is contained in Section~\ref{sec:proof2}.

\medskip

\noindent{\bf Notation.}
The set of non-identity elements of a group $G$ is denoted by $G^\#$.

For a subset $X \subseteq G$, the set $\{x^{-1} : x\in X\}$ is denoted by $X^{-1}$ and
the subgroup generated by $X$ is denoted by $\langle X\rangle$.
The element $\sum_{x \in X} {x}$ of the group ring $\Z G$ is denoted by
 $\underline{X}$.

For $L \trianglelefteqslant G$, the canonical
epimorphism from $G$ to $G/L$ is denoted by $\pi_{G/L}$.

The group of all permutations of a set $\Omega$ is denoted by
$\sym(\Omega)$ and the identity element of $\sym(\Omega)$ by $\mathrm{id}_\Omega$.

For $A \leqslant \sym(\Omega)$ and $\alpha \in \Omega$,
the stabiliser of $\alpha$ in $A$ is denoted by $A_{\alpha}$, the orbit of $\alpha$
under $A$ by $\alpha^A$, and the set of all orbits under $A$ by $\orb(A,\Omega)$.

The right regular representation of $G$ is denoted by
$\rho_G$, i.e., for $x, y \in G$, $x^{\rho_G(y)}=xy$.
The image $\rho_G(G)$ is also denoted by $G\r$.

The set of all permutation groups of $G$ containing
$G\r$ is denoted by $\Sup(G\r)$.

For a set $\Delta \subseteq \sym(G)$ and a section
$S=U/L$ of $G$, we set
$$
\Delta^S=\{\varphi^S:~\varphi \in \Delta,~S^\varphi=S\},
$$
where $S^\varphi=S$ means that $\varphi$ maps $U$ to itself and permutes the
$L$-cosets in $U$ among themselves and
$\varphi^S$ denotes the bijection of $S$ induced by $\varphi$.
\section{S-rings}\label{sec:S-r1}
Let $G$ be a finite group with identity element $e$
and $\Z G$ be the integer group ring.
A subring  $\A \subseteq \Z G$ is called an \emph{S-ring}
(or \emph{Schur ring}) over $G$ if there exists a partition
$\cS(\A)$ of $G$ such that

\begin{enumerate}[{\rm (1)}]
\setlength{\itemsep}{0.4\baselineskip}
\item $\{e\} \in \cS(\A)$,
\item  if $X \in \cS(\A)$ then $X^{-1} \in \cS(\A)$,
\item $\A=\Span_{\Z}\{\underline{X} :\ X \in \cS(\A)\}$.
\end{enumerate}
\smallskip

The elements in $\cS(\A)$ are called the \emph{basic sets} of
$\A$ and the number $\rk(\A):=|\cS(\A)|$ is called the \emph{rank} of
$\A$. The definition of an S-ring is due to Wielandt~(see \cite[Chapter~IV]{Wbook}).
The motivation comes from the following result of Schur
(see \cite[Theorem~24.1]{Wbook}).

\begin{thm}
{\rm (\cite{Schur})}
If $A \in \Sup(G\r)$, then the $\Z$-submodule
$\Span_{\Z}\{\underline{X} : X \in \orb(A_e,G)\}$
is a subring of $\Z G$.
\end{thm}

Clearly, the ring in the theorem is an example of an S-ring, also
called the \emph{transitivity module} over $G$ induced by $A$ and denoted by
$V(G,A_e)$.  An S-ring $\A$ is called \emph{schurian} if $\A=V(G,A_e)$ for some permutation group $A \in \Sup(G\r)$. We remark that not all S-rings are schurian
(see \cite{Wbook}).
In the particular case when $A=G\r K$ for some subgroup
$K \leqslant  \aut(G)$, the S-ring $V(G,A_e)$ is called \emph{cyclotomic} and also
denoted by $\cyc(K,G)$. In this case the basic sets are the orbits under $K$.

Let $\A$ be an S-ring over a group $G$. A set $X \subseteq G$ is  called an
\emph{$\A$-set} if $\underline{X} \in \A$, and a subgroup
$H \leqslant G$ is called an \emph{$\A$-subgroup} if
$\underline{H} \in \A$.
The S-ring $\A$ is \emph{primitive} if $G$ contains no non-trivial
proper $\A$-subgroup.
Suppose that $\A=V(G,A_e)$ for some permutation group
$A \in \Sup(G\r)$. Then $H \leqslant G$ is an $\A$-subgroup
if and only if the partition of $G$ into its right $H$-cosets is
$A$-invariant. Hence $A$ is primitive if and only if so is
$V(G,A_e)$.

\begin{prop}\label{W}
{\rm (\cite{W})} Suppose that $G$ is an abelian group of composite order having a cyclic
Sylow subgroup. Then every primitive S-ring over $G$ is of rank $2$.
\end{prop}

For a subset $X \subseteq G$ and integer $m$, define
$X^{(m)}=\{ x^m : x \in X\}$; and for a group ring
element $\eta=\sum_{g \in G}c_g g$, define
$\eta^{(m)}=\sum_{g \in G}c_g g^m$.
Two useful properties of S-rings over abelian groups are invoked next.
The statement in part (i) is \cite[Theorem~23.9(a)]{Wbook} and the statement in part (ii)  follows from the proof of \cite[Theorem~23.9(b)]{Wbook}.
For an abelian group $G$ and a prime divisor $p$ of the order of $G$ we will use the notation $G[p]=\{g\in G : g^p=e \}$.

\begin{prop}\label{SW}
{\rm (\cite{Wbook})}
Let $\A$ be an S-ring over an abelian group $G$.
\begin{enumerate}[{\rm (i)}]
\item If $m$ is an integer coprime to $|G|$ and $\eta \in \A$, then
$\eta^{(m)} \in \A$. In particular, $X^{(m)} \in \cS(\A)$ whenever $X \in \cS(\A)$.
\item If $p$ is a prime divisor of $|G|$, $1 \leqslant k \leqslant p-1$
and $X \subseteq G$ is an $\A$-set, then the set
$$
X^{[p,k]}:=\{ x^p : x \in X~\text{and}~|X \cap xG[p]| \equiv k(\mathrm{mod}~p) \}
$$
is an $\A$-set (possibly empty). Hence the set
$$
X^{[p]}:=\{ x^p : x \in X~\text{and}~|X \cap xG[p]| \not\equiv 0(\mathrm{mod}~p) \}
$$
is also an $\A$-set.
\end{enumerate}
\end{prop}

Let $G$ be an arbitrary group and $\A$ be an S-ring over $G$.
With each $\A$-set $X$ one can naturally associate two
$\A$-subgroups, namely, $\langle X \rangle$ and
$$
\rad(X):=\{g\in G:\ gX=Xg=X\}.
$$

Let $L \trianglelefteqslant U \leqslant G$. The section $U/L$ is called an \emph{$\A$-section} if $U$ and $L$ are $\A$-subgroups. If $S=U/L$ is an $\A$-section, then the module
$$
\A_S:=\Span_{\Z}\{ \underline{X^{\pi_{U/L}}}:~X \in \cS(\A),~
X \subseteq U \}
$$
is an S-ring over $S$.
Note that, if $\A=V(G,A_e)$ and $S$ is an $\A$-section,
then $\A_S=V(S,(A^S)_{e_S})$ and so $\A_S$ is schurian (see \cite[Proposition~2.8]{HM}). Here $e_S$ denotes the identity element of $S$.

Let $\A$ be an S-ring over a group $G$ and $\B$ be an
S-ring over a group $H$. A bijection $\varphi : G \to H$ is
called an \emph{isomorphism} from $\A$ to $\B$ if
$\rk(\A)=\rk(\B)=r$,
and there is an ordering $X_1,\ldots,X_r$ of the basic sets
in $\cS(\A)$ and an ordering $Y_1,\ldots,Y_r$ of the basic sets  in
$\cS(\B)$ such that $\varphi$ is an isomorphism from
$\cay(G,X_i)$ to $\cay(H,Y_i)$ for every $1 \leqslant i \leqslant r$.
If there is an isomorphism from $\A$ to $\B$, then we say
that $\A$ and $\B$ are \emph{isomorphic} and write $\A \cong \B$.
Let $\iso(\A,\B)$ denote the set of all isomorphisms from $\A$ to $\B$.
An isomorphism $\varphi \in \iso(\A,\B)$ is called \emph{normalised} if it maps
the identity element $e_G$ to the identity element $e_H$.
If $\varphi$ is normalised, then
$X_i^{\,\varphi} \in \cS(\B)$ for every basic set $X_i \in \cS(\A)$ and $\varphi$
also satisfies the condition:
\begin{equation}\label{eq:XiXj}
\forall  1 \leqslant i, j \leqslant r:~(X_i X_j)^\varphi=X_i^{\,\varphi}X_j^{\,\varphi}.
\end{equation}

Some further properties are collected below.

\begin{prop}\label{iso}
{\rm (\cite[Proposition~2.7]{HM})}
Let $\varphi : \A \to \B$ be a normalised isomorphism from an
S-ring $\A$ over a group $G$ to an S-ring $\B$ over a group $H$,
and let $E \leqslant G$ be an $\A$-subgroup.
\begin{enumerate}[{\rm (i)}]
\item The image $E^\varphi$ is a $\B$-subgroup of $H$. Moreover, the restriction
$\varphi_E : E \to E^\varphi$ is an isomorphism between $\A_E$ and
$\B_{E^\varphi}$.
\item For each $x \in G$, $(Ex)^\varphi=E^\varphi x^\varphi$.
\item If $E \trianglelefteqslant G$ and $E^\varphi \trianglelefteqslant H$,
then the mapping $\varphi^{G/E} : G/E \to H/E^\varphi$,
defined by $(Ex)^{\varphi^{G/E}}=E^\varphi x^\varphi$ is a normalised
isomorphism between $\A_{G/E}$ and $\B_{H/E^\varphi}$.
\end{enumerate}
\end{prop}

We are interested in isomorphisms between S-rings over the same group and set
$$
\iso(\A)=\bigcup_{\B~\text{is an S-ring} \atop \text{over}~G}
\iso(\A,\B)~\text{and}~
\iso_e(\A)=\{\varphi \in \iso(\A) : e^\varphi=e\}.
$$
Clearly, $\iso(\A,\A)$ is a subgroup of $\sym(G)$, which contains the
normal subgroup defined as
$$
\aut(\A)=\bigcap_{X\in \cS(\A)}\aut(\cay(G,X)).
$$
This is called the \emph{automorphism group} of $\A$ (see~\cite{KP}).
Clearly, $G\r \leqslant \aut(\A)$.
\section{DCI-groups and CI-S-rings}\label{sec:S-r2}
Babai~\cite{B} gave the following group theoretical
criterion for a subset $X \subseteq G$ to be a CI-subset.

\begin{prop}\label{B}
{\rm (\cite[Lemma~3.1]{B})}
A subset $X \subseteq G$ is a CI-subset if and only if
any two regular subgroups of $\aut(\cay(G, X))$ isomorphic to $G$
are conjugate in $\aut(\cay(G,X))$.
\end{prop}

Let $A, B \in \Sup(G\r)$ such that $A \leqslant B$. Then
$A$ is said to be a
$\mathit{G\r}$\emph{-complete subgroup} of $B$, denoted by
$A \preceq_{G} B$, if for every $\varphi \in \sym(G)$, the inclusion
$(G\r)^\varphi \leqslant B$ implies
$(G\r)^{\varphi\psi} \leqslant A$ for some $\psi \in B$
(see \cite[Definition~2]{HM}). Notice that, the relation
$\preceq_{G}$ is a partial order on $\Sup(G\r)$. In this context
Proposition~\ref{B} reads as
\begin{equation}\label{eq:B2}
X \subseteq G~\text{is a CI-subset} \iff
G\r \preceq_{G} \aut(\cay(G,X)).
\end{equation}

Let $A \leqslant \sym(G)$. The $\mathit{2}$\emph{-closure} $A^{(2)}$ is
the largest permutation group of $G$ satisfying
$$
\orb(A^{(2)},G\times G)=\orb(A,G\times G),
$$
where the groups $A^{(2)}$ and $A$ act on $G \times G$
coordinate-wise. The group $A$ is called
$\mathit{2}$\emph{-closed} if $A^{(2)}=A$.
If $\A=V(G,A_e)$, then $\aut(\A)=A^{(2)}$.
It is well-known that $\aut(\cay(G,X))$ is $2$-closed for
any subset $X \subseteq G$. It follows from this and
\eqref{eq:B2} that $G$ is a DCI-group if
$G\r \preceq_{G} A$ for every $2$-closed permutation group $A \in \Sup(G\r)$.

\begin{prop}\label{HM}
{\rm (\cite[Theorem~2.6]{HM})}
Let $A \in \Sup(G\r)$ be a $2$-closed permutation group and $\A=V(G,A_e)$.
Then the following statements are equivalent.
\begin{enumerate}[{\rm (i)}]
\item $G\r \preceq_{G} A$.
\item $\iso(\A)=\aut(\A) \aut(G)$.
\item $\iso_e(\A)=\aut(\A)_e \aut(G)$.
\end{enumerate}
\end{prop}

An S-ring $\A$ over $G$ is called a \emph{CI-S-ring} (or
\emph{CI} for short) if $\aut(\A) \aut(G)=\iso(\A)$ (see \cite[Definition~3]{HM}).

\begin{prop}
{\rm (\cite[Proposition~2.4]{SM})}
Let $A,\, B \in \Sup(G\r)$ such that $B \leqslant A$, $\A=V(G,A_e)$ and
$\B=V(G,B_e)$. If $B \preceq_{G} \aut(\A)$ and $\B$ is CI, then $\A$ is also CI.
\end{prop}

This allows us to consider only minimal elements of the poset
$(\Sup(G\r),\preceq_{G})$. The set of such
elements will be denoted by $\Sup^{\min}(G\r)$.

\begin{cor}\label{ci}
If $V(G,A_e)$ is CI for every $A \in \Sup^{\min}(G\r)$, then $G$ is a DCI-group.
\end{cor}

In fact, we are going to derive 
Theorems~\ref{main1} and \ref{main2} by
showing that the condition in Corollary~\ref{ci} holds whenever $G$ is one of
the groups in the cited theorems.
\medskip

We conclude the section with a useful lemma.

\begin{lem}\label{sylow}
Let $G$ be an abelian group, $A \in \Sup^{\min}(G\r)$ and $\A=V(G,A_e)$.
Suppose that there exist
$\A$-subgroups $L < U \leqslant G$
such that $|U/L|=np^t$ for a prime $p$ and $1 < n < p$.
Then $LU_p$ is an $\A$-subgroup, where $U_p$ is the
Sylow $p$-subgroup of $U$.
\end{lem}
\begin{proof}
Let $F_1:=G\r,\, F_2, \ldots, F_k$ be a complete set of representatives of the
conjugacy classes of regular subgroups of $A$ isomorphic to $G$.
Then $A=\langle F_1, \ldots, F_k \rangle$ because
of $A \in \Sup^{\min}(G\r)$.
For an easier notation, write $\overline{B}$ for $B^{G/L}$, where
$B \leqslant A$ is any subgroup, and $\bar{e}$ for the identity element of $G/L$. Furthermore,
denote by $\overline{B}_{\, \{U/L\}}$ the setwise stabiliser of
$U/L$ in $\overline{B}$.

For $1 \leqslant i \leqslant k$, let $P_i$ be the Sylow $p$-subgroup of
$(\overline{F_i})_{\{U/L\}}$ and $P$ be a Sylow $p$-subgroup of $\overline{A}_{\, \{U/L\}}$ such that $P_1 \leqslant P$.
It follows from $|U/L|=np^t$ and $n < p$ that a Sylow $p$-subgroup of
$\sym(U/L)$ has $n$ orbits, each containing $p^t$ elements
(see \cite[Example~2.6.1]{DM}). On the other hand, acting on $U/L$, the orbits under $P_1$ are equal to the cosets of $LU_p/L$ in $U/L$, and therefore,
$\orb(P_1,U/L)=\orb(P,U/L)$.

Fix $i$, $2 \leqslant i \leqslant k$.
By Sylow's theorem $P_i^{\delta_i} \leqslant P$ for some
$\delta_i \in \overline{A}_{\, \{U/L\}}$.
Using also that $\overline{F_i}$ is abelian, we find that the partition of $U/L$ into its $LU_p/L$-cosets
is $\overline{F_i}^{\delta_i}$-invariant. Thus it is also
$D$-invariant for $D:=\langle \overline{F_1}, \overline{F_2}^{\delta_2},
\ldots,\overline{F_k}^{\delta_k}\rangle$.
In other words, $\underline{LU_p/L} \in V(G/L,D_{\bar{e}})$.

Let $\gamma_i$ be a preimage of $\delta_i$ under the epimorphism $A \to
\overline{A}$.
Then
$$
\A=V(G,A_e)=V(G,\langle  F_1,\ldots,F_k \rangle_e)=
V(G,\langle F_1,F_2^{\gamma_2},\ldots, F_k^{\gamma_k} \rangle_e),
$$
and so
$$
\A_{G/L}=V\big(G/L,
\overline{\langle F_1,F_2^{\gamma_2}, \ldots, F_k^{\gamma_k}
\rangle}_{\bar{e}})=
V\big(G/L,\langle \overline{F_1}, \overline{F_2}^{\delta_2}, \ldots,
\overline{F_k}^{\delta_k} \rangle_{\bar{e}})=V(G/L,D_{\bar{e}}).
$$
This shows that $\underline{LU_p/L} \in \A_{G/L}$, implying that
$\underline{LU_p} \in \A$.
\end{proof}
\section{Products of CI-S-rings}\label{sec:S-r3}
In this section we review the star and the
generalised wreath product of S-rings. The former was
introduced by Hirasaka and Muzychuk~\cite{HM} and the latter by
Evdokimov and Ponomarenko~\cite{EP} and independently by
Leung and Man~\cite{LM} under the name  wedge product.
\subsection{Star product}
Let $\A$ be an S-ring over a group $G$ and
$V, W \leqslant G$ be two $\A$-subgroups. The S-ring
$\A$ is the \emph{star product} of $\A_V$ with $\A_W$,
written as $\A=\A_V \star \A_W$, if
\begin{enumerate}[{\rm (1)}]
\setlength{\itemsep}{0.4\baselineskip}
\item $V \cap W \lhd W$,
\item every $X \in \cS(\A)$, $X \subseteq W \setminus V$ is a union of some
$(V \cap W)$-cosets,
\item for every $X \in \cS(\A)$ with $X \subseteq G\setminus
(V \cup W)$, there exist basic sets $Y, Z \in \cS(\A)$
such that $X=Y Z$, $Y \subseteq V$ and $Z \subseteq W$.
\end{enumerate}
\smallskip

The star product is \emph{non-trivial} if
$V\neq \{e\}, G$.  In the special case when
$V \cap W=\{e\}$ it is also called the \emph{tensor product} and written
as $\A_V \otimes \A_W$.

\begin{prop}\label{KM1}
{\rm (cf. \cite[Proposition~3.2~and~Theorem~4.1]{KM})}
Let $G$ be a direct product of elementary abelian groups, $A \in \Sup(G\r)$ and
$\A=V(G,A_e)$. If $\A=\A_V \star \A_W$ and both S-rings
$\A_V$ and $\A_{W/(V \cap W)}$ are CI, then $\A$ is also CI.
\end{prop}

\begin{cor}\label{ci-tensor}
In particular, if $\A=\A_V \otimes \A_W$ and both S-rings
$\A_V$ and $\A_W$ are CI, then $\A$ is also CI.
\end{cor}

\begin{prop}\label{EKP}
{\rm (\cite[Lemma~2.3.(2)]{EKP})}
Let $G$ be an abelian group and $\A$ be an S-ring over $G$.
Suppose that $G=H_1\times H_2$ with
$\A$-subgroups $H_1,\, H_2$. Then
$\A \supseteq \A_{H_1} \otimes \A_{H_2}$, and the equality is attained
whenever $\A_{H_1} = \Z H_1$ or $\A_{H_2} = \Z H_2$.
\end{prop}

\begin{lem}\label{ci-complementary}
Let $G$ be an abelian group,
$A \in Sup^{min}(G_R)$ and $\A=V(G,A_e)$.
Suppose that $G=H_1\times H_2$ with $\A$-subgroups $H_1,\, H_2$.
Then $\A=\A_{H_1} \otimes \A_{H_2}$. If $\A_{H_1}$ and $\A_{H_2}$ are CI, then
$\A$ is also CI.
\end{lem}
\begin{proof}
Let $K_i$ be the kernel of the action of $A$ on the set of $H_i$-cosets where $i=1,2$.  The groups $K_1,\, K_2$ are normal in $A$ and intersect trivially
because $H_1 \cap H_2=\{e\}$.
Pick a regular abelian subgroup $F \leqslant A$.
Then $F=(F \cap K_1) \times (F \cap K_2) \leqslant K_1K_2$. Therefore, any regular abelian subgroup of $A$ is contained in $K_1K_2$, implying that $K_1K_2 \preceq_{G} A$. By
$\preceq_{G}$-minimality of $A$ we conclude that $A=K_1K_2$.

Therefore, the permutation group $A=K_1K_2$ acting on $G=H_1H_2$ is
permutation isomorphic to the permutation direct product
$K_1^{H_1} \times K_2^{H_2}$ acting on $H_1 \times H_2$ (see \cite[p.~17]{DM}).
This implies that $\A =\A_{H_1} \otimes \A_{H_2}$, as required.
If $\A_{H_i}$ and $\A_{H_2}$ are CI, then so is $\A$ by Corollary~\ref{ci-tensor}.
\end{proof}
\subsection{Generalised wreath product}

Let $\A$ be an S-ring over a group $G$ and $S=U/L$ be an
$\A$-section of $G$.  The S-ring $\A$ is the \emph{$S$-wreath product} (also called the \emph{generalised wreath product} of
$\A_U$ with $\A_{G/L}$), written as $\A=\A_U \wr_{S} \A_{G/L}$,  if
\begin{enumerate}[{\rm (1)}]
\setlength{\itemsep}{0.4\baselineskip}
\item $L \trianglelefteqslant G$,
\item every $X \in \cS(\A)$, $X \subseteq G \setminus U$ is union of some $L$-cosets.
\end{enumerate}
\smallskip

The $S$-wreath product is \emph{non-trivial} if $L \neq \{e\}$ and $U \neq G$. 
Notice the following relation with the star product.  
If $A_V \star \A_W$ is defined over the group $G$ such that 
$V \cap W \trianglelefteqslant G$, then the latter star product becomes 
the $V/(V \cap W)$-wreath product.

An S-ring $\A$ is called \emph{decomposable} if it can be
expressed as a non-trivial $S$-wreath product and \emph{indecomposable} otherwise.
In the special case when $U=L$, i.e., $S$ is trivial, the $S$-wreath product is also called \emph{wreath product} and written as $\A_U \wr \A_{G/U}$.

The following result is a sufficient condition for the CI-property
of a generalised wreath product. To state the condition, we set the notation:
$\aut_G(\A):=\aut(\A) \cap \aut(G)$.
Clearly, if $S$ is an $\A$-section of $G$, then $\aut_G(\A)^S \leqslant \aut_S(\A_S)$.

\begin{prop}\label{KR}
{\rm (\cite[Theorem~1.1]{KR})}
Let $G$ be a direct product of elementary abelian groups, and
$\A=\A_U \wr_S \A_{G/L}$ be a non-trivial $S=U/L$-wreath
product such that both $\A_U$ and $\A_{G/L}$ are CI. Then
$\A$ is CI whenever
$$
\aut_S(\A_S)=\aut_U(\A_U)^{S} \aut_{G/L}(\A_{G/L})^S.
$$
\end{prop}

Note that, if $\A_S=\Z S$ in Proposition~\ref{KR}, then
$\aut_S(\A_S)$ is trivial, so we obtain the following.

\begin{cor}\label{ci-wp}
If $\A_S=\Z S$ in Proposition~\ref{KR}, then $\A$ is CI.
\end{cor}

Two subgroups $K_1, K_2 \leqslant \aut(G)$ are
\emph{Cayley equivalent}, written as $K_1 \approx_{\mathrm{Cay}} K_2$,
if $\orb(K_1,G)$ $=\orb(K_2,G)$ (see \cite{KR}).
A cyclotomic S-ring $\A$ over $G$ is said to be \emph{Cayley minimal} if
$$
\big\{ K \leqslant \aut(G) : K \approx_{\mathrm{Cay}}
\aut_G(\A) \big\} = \{\aut_G(\A)\}.
$$

\begin{prop}\label{ci-caymin}
{\rm (\cite[Lemma~4.2]{KR2})}
With the assumptions in Proposition~\ref{KR}, suppose that at least one of the S-rings
$\A_U$ and $\A_{G/L}$ is cyclotomic and $\A_S$ is Cayley minimal.
Then $\A$ is CI.
\end{prop}

This proposition will be especially useful in conjunction with the following lemma.

\begin{lem}\label{mix}
Let $\A$ be an S-ring over a cyclic group $G$ of order $n$.
\begin{enumerate}[{\rm (i)}]
\item If $n$ is a prime, then $\A$ is cyclotomic.
\item If $n=pq$ for distinct primes $p, q$ and
$\rk(\mathcal{A}) \ne 2$, then $\A$ is cyclotomic or
a non-trivial wreath product of two S-rings.
\item If $\A$ is cyclotomic, then it is Cayley minimal.
\end{enumerate}
\end{lem}
\begin{proof}
The statement in (i) follows from Proposition~\ref{SW}(i). The statement in (ii) follows from \cite[Theorem~2.10]{KP}.

For (iii) let $\A=\cyc(K,G)$ for a subgroup $K \leqslant \aut(G)$.
Let $x$ be a generator of $G$ and $X \in \cS(\A)$ be the basic set containing $x$.
It is easy to see that $K$ is regular on $X$, hence $|K|=|X|$.
This implies that $\aut_G(\A)=K$ and $K^\prime \not\approx_{\mathrm{Cay}}
K$ for any proper subgroup $K^\prime < K$, i.e., $\A$ is Cayley minimal.
\end{proof}

Dobson and Witte~\cite{DW} described the groups in $\Sup(G\r)$ where
$G \cong C_p^2$ for a prime $p$ (the description of the imprimitive groups were obtained earlier
by Jones~\cite{J}).  The proposition below follows from their result and for our convenience it is formulated here in the language of S-rings.

\begin{prop}\label{DW}
{\rm (cf.~\cite[Theorem~14]{DW})}
Let $G \cong C_p^2$ for a prime $p$, $A \in \Sup(G\r)$ and
$\A=V(G,A_e)$.
If $G$ contains exactly one $\A$-subgroup of order $p$,
say $L$, then $\A=\A_L \wr \A_{G/L}$.
\end{prop}

An S-ring $\A$ over a group $G$ is a \emph{$p$-S-ring} if
$G$ is a $p$-group and for every $X \in \cS(\A)$, $|X|$ is
equal to a power of $p$.

\begin{prop}\label{KM2}
{\rm (\cite[Lemma~5.2]{KM})}
Let $G$ be an abelian group, $A \in \Sup^{\min}(G\r)$ and $\A=V(G,A_e)$.
Suppose that $U$ is an $\A$-subgroup such that $G/U$ is a $p$-group for a prime $p$.
Then $\A_{G/U}$ is a $p$-S-ring.
\end{prop}

It is obvious that $\Z C_p$ is the only $p$-S-ring over $C_p$. Furthermore, it is
well-known that up to isomorphism, there are two
$p$-S-rings over $C_p^2$, namely
\begin{equation}\label{eq:p2-p}
\Z C_p^2~\text{and}~\Z C_p \wr \Z C_p,
\end{equation}
(see, e.g., \cite[Section~3.1]{HM}.

For the next two propositions let $G$ be an abelian group such
that $q \mid |G|$ and $q^2 \nmid |G|$ for a prime $q$ and let $\A$ be an S-ring over
$G$. Let $Q$ be the least $\A$-subgroup of order divisible by $q$ and
$H$ be the unique maximal $\A$-subgroup of order coprime to $q$.

\begin{prop}\label{MS1}
{\rm (\cite[Corollary~3.2]{SM})}
With notation as above, $\A$ is the $HQ/Q$-wreath product.
\end{prop}

\begin{prop}\label{MS2}
{\rm (\cite[Propositions~3.4~and~3.5]{SM})}
With notation as above, if
$|HQ/H| \ne q$ or $\A_{HQ/H} \cong \Z C_q$,
then $\A_{HQ}=\A_H \star \A_Q$.
\end{prop}

\begin{lem}\label{ci-gwp}
With the assumptions in Proposition~\ref{KR}, $\A$ is CI whenever
$\A_{G/L}=\A_S \otimes \A_H$ for some $\A_{G/L}$-subgroup $H < G/L$.
\end{lem}

\begin{proof}
The following containment is clear:
$$
\aut_{G/L}(\A_{G/L})^S \geqslant (\aut_S(\A_S) \times
\aut_H(\A_H))^S =\aut_S(\A_S).
$$
On the other hand, $\aut_{G/L}(\A_{G/L})^S\leqslant \aut_S(\A_S)$ and therefore,
$\aut_{G/L}(\A_{G/L})^S=\aut_S(\A_S)$. Then $\A$ is CI by Proposition~\ref{KR}.
\end{proof}

\begin{lem}\label{centre}
Let $G$ be an abelian group, $A \in \Sup^{\min}(G\r)$ and $\A=V(G,A_e)$.
Suppose that $\A$ is indecomposable and
$L$ is an $\A$-subgroup of prime order. Then
$\rho_G(L) \leqslant Z(A)$. Moreover, for each $u \in L$, $\{u\} \in \cS(\A)$.
\end{lem}
\begin{proof}
Let $p=|L|$ and write $\hat{L}=\rho_G(L)$.
Let $K$ be the kernel of the action of $A$ on the set of $L$-cosets
in $G$.
For $x \in G$, let $K_{Lx}$ denote the pointwise stabiliser of
$Lx$ in $K$. Define the binary relation $\sim$ on the set of
$L$-cosets in $G$ by $Lx \sim Ly$ if and only if $K_{Lx}=K_{Ly}$.
It is obvious that $\sim$ is an equivalence relation.
Also, for arbitrary $\gamma \in A$, $K_{(Lx)^\gamma}=(K_{Lx})^\gamma$, implying that $\sim$ is also $A$-invariant.
This shows that the set $\{Lx : Lx \sim L\}$ is a block for $A$ acting on
the set of $L$-cosets in $G$. Consequently, the set
$U:=\bigcup_{L_x \sim L}Lx$ is a block
for $A$ acting on $G$, and so $U$ is an $\A$-subgroup.
Clearly, $L \leqslant U$.

Let $\gamma \in K_{Lx}$ for some $x \in U$. By the definition of $U$,
$\gamma \in K_U$, the pointwise stabiliser of $U$ in $K$.
In other words, $K^U$ is faithful on $Lx$ for every $x \in U$.

Assume for the moment that $U < G$.
Let $x \notin U$. The group $K$ acts primitively on $Lx$ because $|Lx|=p$, and
$K_L \trianglelefteqslant K$.
Since $x \notin U$, $L \not\sim Lx$, and hence $(K_L)^{Lx} \ne 1$.
We obtain that the orbit $x^{K_L}=Lx$, so $L \leqslant \rad(x^{A_e})$.
This shows that $\A$ is the
non-trivial $U/L$-wreath product, a contradiction.
Thus $U=G$.
As $K=K^U$ is faithful on $L$, $\hat{L}$ is the unique
Sylow $p$-subgroup of $K$. On the other hand, if $F \leqslant A$ is
any abelian regular subgroup, then $K \cap F$ has order $p$, and thus
we find $K \cap F=\hat{L}$, in particular, $\hat{L} \leqslant Z(F)$.
This yields $\hat{L} \leqslant Z(A)$ because
$A$ is generated by
its regular subgroups isomorphic to $G$ due to the condition $A \in \Sup^{\min}(G\r)$.

For the second assertion, choose $u \in L$ and let $X \in \cS(\A)$ be the basic
set containing $u$. Then $X=u^{A_e}$ and as $\hat{L} \leqslant Z(A)$, we obtain
$X=e^{\rho_G(u)A_e}=(e^{A_e})^{\rho_G(u)}=\{u\}$.
\end{proof}
\section{Proof of Theorem~\ref{main1}}\label{sec:proof1}
Throughout this section we keep the following notation:
\begin{quote}
$G \cong C_p \times C_n$ for a prime $p$ and a square-free number $n$
(that may or may not be divisible by $p$),
$A \in \Sup^{\min}(G\r)$ and $\A=V(G,A_e)$.
\end{quote}

In view of Corollary~\ref{ci}, it is sufficient to show that
$\A$ is CI.
We proceed by induction on the total number of prime divisors (counted with multiplicities) of $|G|$,
that we will denote by $\Omega(|G|)$.

If $\Omega(|G|)=1$, then $G \cong C_p$. It follows from
Proposition~\ref{B} via Sylow's theorem
that $A=G\r$, hence $\A=\Z G$, which is clearly CI.

Assume that $\Omega(|G|) > 1$ and the assertion
holds for any group $\tilde{G} \cong C_{\tilde{p}} \times
C_{\tilde{n}}$, where $\tilde{p}$ is a prime, $\tilde{n}$ is square-free and
$\Omega(|\tilde{G}|) < \Omega(|G|)$.
Note that, this implies that every schurian S-ring over $\tilde{G}$ is CI.

If $G \cong C_p^2$, then
$\A \cong \Z C_p^2$ or $\Z C_p \wr \Z C_p$
by Proposition~\ref{KM2} and \eqref{eq:p2-p}.
In either case, $\A$ is CI (for the latter S-ring, see Corollary~\ref{ci-wp}).

Now, let $n_{p^\prime} > 1$, where $n_{p^\prime}$ denotes
the $p^\prime$-part of $n$. Let
$n_{p^\prime}=q_1 \cdots q_k$ be the
prime decomposition of $n_{p^\prime}$, $P$ be the Sylow $p$-subgroup and $C$ be
the Hall $p^\prime$-subgroup of $G$.  Then $P \cong C_p$ or $C_p^2$ and
$C \cong C_{n_{p^\prime}}$.
For $1 \leqslant i \leqslant k$, let $Q_i$ be the least $\A$-subgroup of $G$ of
order divisible by $q_i$, and $H_i$ be the
unique maximal $\A$-subgroup of order coprime to $q_i$.
\medskip

\noindent{\bf Claim~1.}
$\A$ is CI, unless $H_iQ_i \ne G$ for every $1 \leqslant i \leqslant k$.
\medskip

Suppose that $H_iQ_i=G$ for some $1 \leqslant i \leqslant k$.
Then $\A=\A_{H_i} \star \A_{Q_i}$.
This follows from Proposition~\ref{MS2} if $|G/H_i| \ne q_i$, and from
Propositions~\ref{KM2}~and~\ref{MS2} if $|G/H_i|=q_i$.

If $Q_i/(Q_i \cap H_i) \not\cong G$, then the induction hypothesis guarantees that both
$\A_{H_i}$ and $\A_{Q_i/(Q_i \cap H_i)}$ are CI, and hence so is $\A$ by  Proposition~\ref{KM1}.

Let $Q_i/(Q_i \cap H_i) \cong G$.
Then $Q_i=G$, $H_i=\{e\}$, and these imply that
$\A$ is primitive. By Proposition~\ref{W}, $\rk(\A)=2$,
and so $\A$ is CI in this case as well. This completes the proof of Claim~1.
\medskip

\noindent{\bf Claim~2.}  $\A$ is CI, unless $C$ is
an $\A$-subgroup and $\A_{G/C} \cong \Z C_p^2$.
\medskip

In view of Claim~1, we may assume that
$H_iQ_i \ne G$ for every $1 \leqslant i \leqslant k$.
Then $\A$ is the non-trivial $H_iQ_i/Q_i$-wreath product by Proposition~\ref{MS1}.

Assume first that $P \cong C_p$, i.e., $G$ is a cyclic group.
Let $X \in \cS(\A)$ be a basic set containing a generator of $G$, say $x$, and let
$V=\rad(X)$. Then for every $1 \leqslant i \leqslant k$, $x \notin H_iQ_i$, and so
$V \geqslant Q_i$. We obtain $V=C$, in particular,
$\underline{C} \in \A$. By Proposition~\ref{KM2}, $\A_{G/V} \cong \Z C_p$, and it
follows that $X=Vx$. This and Proposition~\ref{SW}(i)
imply that $\A=\A_V \wr \A_{G/V}$, hence $\A$ is CI by Corollary~\ref{ci-wp}.

Now, suppose that $P \cong C_p^2$ and let $c$ be a generator of $C$.

Assume for the moment that some cyclic subgroup of index $p$ is not an $\A$-subgroup, i.e.,
$\underline{\langle xc \rangle} \notin \A$ for some $x \in P^\#$.
Let $X \in \cS(\A)$ be the basic set containing $xc$ and let $V=\rad(X)$.
If $xc\in H_iQ_i$ for some $1 \leqslant i \leqslant k$, then $\langle xc \rangle=H_iQ_i$
because $|\langle xc \rangle|=|G|/p$ and $H_iQ_i < G$,
and this contradicts that $\underline{\langle xc \rangle} \notin \A$.
Thus $xc \notin H_iQ_i$ for every $1 \leqslant i \leqslant k$, and one finds as above
that $V \geqslant C$. If $V=C$, then $\underline{C} \in \A$.
The basic set  $X/V \in \cS(\A_{G/V})$ satisfies $|\rad(X/V)|=1$.
On the other hand, $\A_{G/V}=\A_{G/C} \cong \Z C_p^2$ or
$\Z C_p \wr \Z C_p$ by Proposition~\ref{KM2} and \eqref{eq:p2-p}.
We conclude that $X=Vx$, and so $\underline{\langle xc \rangle}=\underline{\langle X \rangle} \in \A$, which is impossible. Thus $V > C$, and it can be deduced from this
in the same way as before that $\A=\A_V \wr \A_{G/V}$, so $\A$ is CI.

To sum up, $\A$ is CI, unless $\underline{\langle xc \rangle} \in \A$
for every $x \in P^\#$. It is easy to see that this implies
$\underline{C} \in \A$ and $\A_{G/C} \cong \Z C_p^2$.
\medskip

\noindent{\bf Claim~3.}  $\A$ is CI.
\medskip

In view of Claim~2, we may assume that $\underline{C} \in \A$ and
$\A_{G/C} \cong \Z C_p^2$.
By Proposition~\ref{HM}, $\A$ is CI exactly when
$\iso_e(\A)=\aut(\A)_e \aut(G)$.
Let $\varphi \in \iso_e(\A)$. We finish the proof of Claim~3 by finding an automorphism
$\alpha \in \aut(G)$ such that
\begin{equation}\label{eq:Xalpha}
\forall X \in \cS(\A):~X^\varphi=X^\alpha.
\end{equation}

By Proposition~\ref{iso}(i), $C^\varphi$ is a subgroup of $G$ of order
$n_{p^\prime}$. Thus $C^\varphi=C$, and the restriction
$\varphi_C$ induces a normalised isomorphism of $\A_C$, see Proposition~\ref{iso}(i).
Furthermore, $\varphi^{G/C}$ is  a normalised isomorphism of $\A_{G/C}$ defined
in Proposition~\ref{iso}(iii).
Since both $\A_C$ and $\A_{G/C}$ are schurian, these are also CI by the induction hypothesis. Thus there exists
$\alpha_1 \in \aut(C)$ such that
\begin{equation}\label{eq:a1}
\forall X \in \cS(\A), X \subseteq C:~ X^{\varphi_C}=X^{\alpha_1}.
\end{equation}
Also, there exists $\alpha_2 \in \aut(G/C)$ such that
\begin{equation}\label{eq:a2}
\forall X \in \cS(\A):~ (X^{\pi_{G/C}})^{\varphi^{G/C}}=
(X^{\pi_{G/C}})^{\alpha_2}.
\end{equation}
Since $G \cong C \times G/C$, there exists a unique automorphism
$\alpha \in \aut(G)$ such that
$\alpha^C=\alpha_1$ and $\alpha^{G/C}=\alpha_2$.
We claim that $\alpha$ satisfies the condition in Eq.~\eqref{eq:Xalpha}.

If $X \in \cS(\A)$, $X \subseteq C$, then by Eq.~\eqref{eq:a1}, $X^\varphi=X^{\varphi_C}=X^{\alpha_1}=X^\alpha$.

Let $X \in \cS(\A)$, $X \not\subseteq C$. Since $\A_{G/C} \cong \Z C_p^2$,
$X \subseteq Cx$ for some element $x \in P^\#$.
Let $U=\langle C, x \rangle$. Then $U=\langle X \rangle C$,
showing that $\underline{U} \in \A$.
Let $P_1$ be the minimal $\A$-subgroup contained in $U$ whose order is divisible by
$p$. By Proposition~\ref{MS2}, $\A_U=\A_C \star \A_{P_1}$. Moreover,
letting $D=P_1 \cap C$,
the basic set $X$ can be written in the form
\begin{equation}\label{eq:x}
X=YDx,~Y \in \cS(\A),~Y \subseteq C.
\end{equation}

Assume first that $Y=\{e\}$ in Eq.~\eqref{eq:x}, i.e., $X=Dx$.
Let $\psi=\varphi \alpha^{-1} \in \sym(G)$. Then $\psi \in \iso_e(\A)$.
Using Proposition~\ref{iso}(ii)--(iii) and Eq.~\eqref{eq:a2}, we can write
$$
(Cx)^\varphi=Cx^\varphi=(Cx)^{\varphi^{G/C}}=(Cx)^{\alpha_2}=(Cx)^\alpha.
$$
This shows that $(Cx)^\psi=Cx$ and so $U^\psi=U$. Thus $P_1^\psi \leqslant U$,
and as $|P_1^{\,\psi}|=|P_1|$ also holds, $P_1^{\,\psi}=P_1$. We conclude
$$
(Dx)^\psi=(Cx \cap P_1)^\psi=Cx \cap P_1=Dx.
$$
Equivalently, $(Dx)^\varphi=(Dx)^\alpha$.

Finally, let $Y \ne \{e\}$ in Eq.~\eqref{eq:x}. Then by Eq.~\eqref{eq:XiXj}, $X^\varphi=(YDx)^\varphi=Y^\varphi (Dx)^\varphi=Y^\alpha (Dx)^\alpha=
X^\alpha$. This completes the proof of Claim~3 as well as the proof of Theorem~\ref{main1}.
\section{Primitive rational S-rings over $C_p^2 \times C_q^2$ and translation nets}
\label{sec:S-r4}
Let $G$ be an abelian group and $\exp(G)$ be its exponent.
Let $\P(G)$ be the subgroup of $\aut(G)$ consisting of
the power automorphisms
$$
\pi_m : x \mapsto x^m,~x\in G,
$$
where $1 \leqslant m \leqslant \exp(G)$ and $\gcd(m,\exp(G))=1$.

The \emph{trace} $\trX$ of a subset $X \subseteq G$ is defined by
$\trX=\bigcup_{\pi_m\in \P(G)}X^{\pi_m}$.
The cyclotomic S-ring $\cyc(\P(G),G)$ is also known as the \emph{complete S-ring of traces}
over $G$ and denoted by $W(G)$.
If $\A$ is an S-ring over $G$, then its \emph{rational closure} $\trA$ is defined by
$\trA=\A \cap W(G)$. The S-ring $\A$ is called \emph{rational} if $\trA=\A$,
i.e., $\A \subseteq W(G)$.  In this case $\trX=X$ holds for every basic set
$X\in \cS(\A)$. We also say that $X$ is rational if $\trX=X$.

\begin{lem}\label{rank2}
\begin{enumerate}[{\rm (i)}]
\item Let $\A$ be an S-ring over the abelian group $G$ and $X\in \cS(\A)$ be a basic set.
If $X$ contains elements of coprime orders, then $X$ is rational.
\item Let $G$ be an abelian group whose order is divisible by at least two distinct primes
and let $\A$ be an S-ring over $G$. If $\trA$ is of rank 2, then so is $\A$.
\end{enumerate}
\end{lem}
\begin{proof}
(i): Assume that $x_1,x_2\in X$ have coprime orders. Then we can write $G$ as a
direct product $G=G_1\times G_2$, where $x_1\in G_1$, $x_2\in G_2$ and $\gcd(|G_1|,|G_2|)=1$.
Let $m$ be an integer such that $\gcd(m,\exp(G))=1$. We have to show that $X^{\pi_m}=X$.
By the Chinese remainder theorem we can find $m_1$ and $m_2$ satisfying
\begin{gather*}
m_1 \equiv 1 \pmod {\exp(G_1)}, \quad  m_1 \equiv m \pmod {\exp(G_2)},\\
m_2 \equiv m \pmod {\exp(G_1)}, \quad  m_2 \equiv 1 \pmod {\exp(G_2)}.
\end{gather*}
Then  $m_1m_2 \equiv m \pmod {\exp(G)}$. By Proposition~\ref{SW}(i) we have that
$X^{\pi_{m_1}}, X^{\pi_{m_2}}\in \cS(\A)$. Since $x_1\in X\cap X^{\pi_{m_1}}$ and
$x_2\in X\cap X^{\pi_{m_2}}$ we obtain $X^{\pi_{m_1}}=X^{\pi_{m_2}}=X$, hence
$X^{\pi_m}=X$ as well.

(ii): Let $\{e\}\ne X\in \cS(\A)$. Then $\trX=G^\#$, hence $X$ contains elements of
every prime order occurring in $G$. By (i), it follows that $X=G^\#$.
\end{proof}

\medskip

For the next lemma we define a particular subgroup of $\P(G)$.
If $p$ is a prime divisor of $|G|$, then let
$$
\P_p(G)=\{\pi_m \in \P(G) : m \equiv 1 \pmod {\exp(G)_{p^\prime}}\}.
$$

\begin{lem}\label{trivial}
Let $G$ be an abelian group with Sylow $p$-subgroup $G_p\cong C_p^2$ and assume that $G_p\ne G$.
Let $\A$ be a primitive S-ring over $G$, $X \in \cS(\A)$ a $\P_p(G)$-invariant basic set
and $x \in G^\#$ a $p^\prime$-element. Then $X \cap G_px$ is one of the following sets:
$\emptyset$, $Rx$ or $(G_p\setminus R)x$ for a subgroup $R \leqslant G_p$ of order $p$, $G_p$.
\end{lem}
\begin{proof}
Consider the set $X^{[p]}$. It is contained in $G_{p^\prime}$,
the Hall $p^\prime$-subgroup of $G$, and
by Proposition~\ref{SW}(ii), it is an $\A$-set. If $p \nmid |X \cap G_px|$, then
$\langle X^{[p]} \rangle$ is a non-trivial proper $\A$-subgroup.
But this is impossible as $\A$ is primitive, hence $p \mid |X \cap G_px|$.
Now $X \cap G_px$ is mapped to itself by every automorphism in $\P_p(G)$, hence
$X \cap G_px=(X\cap\{x\})\cup \bigcup_{i=1}^m R_i^\#x$ with some $p$-element subgroups
$R_1,\dots,R_m\leqslant G_p$ ($0\leqslant m\leqslant p+1$). Thus
$|X \cap G_px|=f+m(p-1)$, where $f\in\{0,1\}$. As $p \mid |X \cap G_px|$, we obtain
that $(f,m)=(0,0), (1,1), (0,p)$ or $(1,p+1)$, and the result follows.
\end{proof}

We analyse rational S-rings over
$G \cong C_{p_1}^2 \times \cdots \times C_{p_k}^2$, where $p_1,\ldots,p_k$ are
pairwise distinct primes. Clearly,
$$
W(G)=W(G_{p_1})\otimes \cdots \otimes W(G_{p_k}),
$$
where $G_{p_i}$ is the Sylow $p_i$-subgroup of $G$.
The  basic sets of $W(G_{p_i})$ distinct from $\{e\}$
are in one-to-one correspondence with the $p_i+1$ proper non-trivial subgroups of $G_{p_i}$, denoted by $L_{i,1},\ldots,L_{i,p_i+1}$.
The basic set corresponding to $L_{i,j}$
is $X_{i,j}:=L_{i,j}^\#$. Furthermore, we set the notation $X_{i,0}$ for the basic set
$\{e\}$. Writing $[0,n]$ for $\{0,1,\ldots,n\}$, we obtain
$$
\cS(W(G))=\Big\{ \prod_{i=1}^{k}X_{i,t_i} : (t_1,\ldots,t_k) \in [0,p_1+1]\times\cdots\times [0,p_k+1] \Big\}.
$$
This shows that $W(G)$ is of rank $\prod_{i=1}^{k}(p_i+2)$.

Let now $\A$ be an arbitrary rational S-ring over $G$, i.e., $\A \subseteq W(G)$.
Every basic set of $\A$ is a union of basic sets of $W(G)$, and for this reason, it is
encoded by a non-empty subset of $[0,p_1+1]\times\cdots\times [0,p_k+1]$.
More precisely, if $T \subseteq [0,p_1+1]\times\cdots\times [0,p_k+1]$, then the corresponding basic set $X$ is given as
\begin{equation}\label{eq:X-rat}
X=\bigcup_{(t_1,\ldots,t_k) \in T}\, \prod_{i=1}^{k}X_{i,t_i}.
\end{equation}

Given an arbitrary subset $T \subseteq [0,p_1+1] \times \cdots \times [0,p_k+1]$,
a $(k-1)$-tuple $a \in \Z^{k-1}$ and a number $1 \leqslant i \leqslant k$, define the subset
$T_i(a) \subseteq [0,p_i+1]$ as follows:
$$
T_i(a)=\{t_i : \exists~t=(t_1,\ldots,t_k) \in T~\text{such that}~
(t_1,\ldots,t_{i-1},t_{i+1},\ldots,t_k)=a\}.
$$

\begin{lem}\label{Ti(a)}
With the notation as above, suppose that $\A$ is primitive,
$k \geqslant 2$ and
$T \subseteq [0,p_1+1] \times\cdots\times [0,p_k+1]$
corresponds to a basic set of $\A$. Then for each non-zero $(k-1)$-tuple
$a \in \Z^{k-1}$ and $1 \leqslant i \leqslant k$,
$$
T_i(a)=\emptyset~\text{or}~\{0,\ell\}~\text{or}~[0,p_i+1]
\setminus \{0,\ell\}~\text{or}~[0,p_i+1]
$$
for some $1 \leqslant \ell \leqslant p_i+1$.
\end{lem}
\begin{proof}
Suppose that $T_i(a) \ne \emptyset$. It follows from Eq.~\eqref{eq:X-rat} that
$a_j \in [0,p_j+1]$ if $1 \leqslant j < i$
and $a_j \in [0,p_{j+1}+1]$ if $i \leqslant j \leqslant k-1$.
Let $X$ be the basic set corresponding to $T$. Then
$$
X \cap G_{p_i} \prod_{j=1}^{i-1}X_{j,a_j} \prod_{j=i}^{k-1} X_{j+1,a_j}
=\bigcup_{\ell \in T_i(a)} \Big( X_{i,\ell} \prod_{j=1}^{i-1}X_{j,a_j} \prod_{j=i}^{k-1} X_{j+1,a_j}\Big).
$$
Thus, for a fixed $x \in \prod_{j=1}^{i-1}X_{j,a_j} \prod_{j=i}^{k-1} X_{j+1,a_j}$,
$X \cap G_{p_i}x=\big(\bigcup_{j \in T_i(a)}X_{i,j}\big) x$.
Then $x \ne e$ because at least one of the entries $a_j$ is non-zero, and Lemma~\ref{trivial} can be applied to $X \cap G_{p_i}x$. We conclude
that, either there exists a subgroup $R\leqslant G_{p_i}$ of order $p_i$ such that
$\bigcup_{j \in T_i(a)} X_{i,j}=R$ or $G_{p_i} \setminus R$, or else 
$X_{i,j}=G_{p_i}$ (recall that, $T_i(a) \ne \emptyset$ was assumed at the beginning of the proof).
\end{proof}

The next theorem is the main result of this section, which is of independent interest.

\begin{thm}\label{primrat}
Suppose that $G \cong C_p^2 \times C_q^2$ for distinct odd primes $p,\, q$ and
$\A$ is a primitive rational S-ring over $G$ of rank at least $3$.
Then $\A$ has a basic set of the form
$$
H_1^\# \cup \cdots \cup H_m^\#,
$$
where each $H_i \leqslant G$ is of order $pq$ and $H_i \cap H_j=\{e\}$ if
$i \ne j$.
\end{thm}
\begin{proof}
We consider the partition of $G^\#$ into the
basic sets of $\A$ distinct from $\{e\}$. Denote the latter basic sets by
$X, Y, Z$, etc. We imagine such partition as a $(p+1) \times(q+1)$ matrix $M$
filled up with the letters $X, Y, Z$, etc. More precisely, let $M_{i,j}=X$ if and only if
$(i,j)\in [1,p+1]\times[1,q+1]$ belongs to the subset of $[0,p+1] \times [0,q+1]$ corresponding to $X$.
Here and in what follows, we use the description of the basic sets of $\A$ established in
Eq.~\eqref{eq:X-rat} with abbreviation $p_1=p$ and $p_2=q$.

Notice that, Lemma~\ref{Ti(a)} implies that
the subsets of $[0,p+1] \times [0,q+1]$
corresponding to the basic sets $X, Y, Z$, etc.~are determined uniquely by $M$ (hence
as well as $\A$). We will freely use symmetries arising by permuting the rows, the columns, the letters, and transposing the matrix.

Suppose that $M_{i,j}=X$ and $M_{i,j^\prime} \ne X$ if $j^\prime\ne j$.
Let $T \subseteq [0,p+1] \times [0,q+1]$ be the subset corresponding to $X$.
Then $T_2(i)=\{0,j\}$ by Lemma~\ref{Ti(a)}. In particular, $(i,0) \in T$, and
so $X_{1,i}X_{2,0}=L_{1,i}^\# \subseteq X$. This shows that $X$ is the only letter in the $i^{\text{th}}$ row occurring exactly once. Applying Lemma~\ref{Ti(a)} again, one concludes that each row and each column is filled up with either a single letter or with the same letter with one exception. We will call the letter that occurs there with at most one exception the \emph{dominant letter} of the row or column.

Assume to the contrary to the claim in the theorem that no basic set of $\A$ is the union of pairwise disjoint
subgroups of order $pq$ without the identity element.
This means that every letter is dominant in at least one row or column.
As $\rk(\A) \geqslant 3$, there are at least two letters. Moreover, primitivity implies that every letter occurs in at least two rows and in at least two columns.

There cannot be three different dominant letters of rows, since in that case there would be a column containing three different letters. As $p,\, q > 2$, we may assume
w.l.o.g.~that the number of rows is at least $6$. Thus, there are $3$ rows with the same dominant letter, say, $X$, hence we see that with at most one exception the dominant letter of the
columns is $X$. Now, since the number of columns is at least $4$, we also conclude that with at most one exception the dominant letter of the rows is $X$.

We show next that there cannot be three different letters. W.l.o.g.~the first $p$ rows are dominated by $X$, the last row by $Y$,
the first $q$ columns by $X$, and the last column by $Z$.
The letter $Y$ should appear somewhere in the first $p$ rows (as it was pointed out,
this follows from primitivity). Since no column is dominated by $Y$, this is the only $Y$ in its column. Hence up to permuting the
rows and columns, the matrix $M$ must look like this:
$$
M={\footnotesize
\begin{pmatrix}
X&X&\dots&X&X&Y \\
X&X&\dots&X&X&Z \\
\vdots&\vdots&\ddots&\vdots&\vdots&\vdots \\
X&X&\dots&X&X&Z \\
Y&Y&\dots&Y&Y&Z
\end{pmatrix}}.
$$
Then $Z$ occurs only in the last column, but primitivity does not allow this. So there are exactly two letters, $X$ and $Y$.
\medskip

\noindent{\bf Case~1.}
Both the last row and column are dominated by $Y$ and all others by $X$.
\medskip

\noindent{\bf Subcase 1A.}  If the entry in the lower right corner is $X$,
then the matrix is uniquely determined:
$$
M=M_1:=
{\footnotesize
\begin{pmatrix}
X&X&\dots&X&X&Y \\
X&X&\dots&X&X&Y \\
\vdots&\vdots&\ddots&\vdots&\vdots&\vdots \\
X&X&\dots&X&X&Y \\
Y&Y&\dots&Y&Y&X
\end{pmatrix}}.
$$

\noindent{\bf Subcase 1B.} If the entry in the lower right corner is $Y$, then up to permuting rows and columns the matrix looks like this:
$$
M={\footnotesize
\begin{pmatrix}
?&X&\dots&X&X&? \\
X&X&\dots&X&X&Y \\
\vdots&\vdots&\ddots&\vdots&\vdots&\vdots \\
X&X&\dots&X&X&Y \\
?&Y&\dots&Y&Y&Y
\end{pmatrix}}
$$
where each question mark should be replaced by
$X$ or $Y$ in such a way that in the first row, as well as in the first column, there can be at most one $Y$.
That allows the following possibilities for the corner elements of the matrix (up to transposition):
$$
M_2:=
{\footnotesize
\begin{pmatrix} X&\dots &X \\ \vdots&&\vdots \\ X&\dots &Y \end{pmatrix}},~
M_3:=
{\footnotesize
\begin{pmatrix} X&\dots &Y \\ \vdots&&\vdots \\ X&\dots &Y \end{pmatrix}},~
M_4:=
{\footnotesize
\begin{pmatrix} X&\dots &Y \\ \vdots&&\vdots \\ Y&\dots &Y \end{pmatrix}},~
M_5:=
{\footnotesize
\begin{pmatrix} Y&\dots &X \\ \vdots&&\vdots \\ X&\dots &Y \end{pmatrix}}.
$$

\noindent{\bf Case 2.} The last row is dominated by $Y$ and all other rows as well as  every column is dominated by $X$.
\medskip

By primitivity, $Y$ must occur somewhere in the first $p$ rows.
No two $Y$'s can be in the same column, because each one is dominated by $X$. Hence up to permutations, we have a unique possibility:
$$
M=M_6:=
{\footnotesize
\begin{pmatrix}
X&X&\dots&X&X&Y \\
X&X&\dots&X&X&X \\
\vdots&\vdots&\ddots&\vdots&\vdots&\vdots \\
X&X&\dots&X&X&X \\
Y&Y&\dots&Y&Y&X
\end{pmatrix}}.
$$

What is left to show is that none of the matrices $M_1,\ldots,M_6$ defines a primitive
S-ring (of rank $3$).
Let $P$ and $Q$ be the Sylow $p$- and $q$-subgroups of $G$, respectively.

We start with the matrix $M_1$.
Let $T \subset [0,p+1] \times [0,q+1]$ be the subset corresponding to $Y$.
Using Lemma~\ref{Ti(a)}, we find
$$
T = \big\{ (i,q+1), (i,0) :  1 \leqslant i \leqslant p \big\}~\cup~
\big\{ (p+1,j), (0,j) : 1 \leqslant j \leqslant q \big\}.
$$
Then due to Eq.~\eqref{eq:X-rat}, 
$Y=L_{1,p+1}(Q \setminus L_{2,q+1})~\cup~(P \setminus L_{1,p+1})L_{2,q+1}$.
It can be seen easily that $Yx=Y$ for every $x \in L_{1,p+1}L_{2,q+1}$.
Thus $\rad(Y)$ is a non-trivial $\A$-subgroup, contradicting that $\A$ is primitive.

Similarly, if $M=M_3, M_4$, or $M_6$ we can find a non-trivial $\A$-subgroup contradicting the primitivity of $\A$.
Namely, if $M=M_3$, then $Y=L_{1,p+1}(Q\setminus L_{2,1})~\cup~(P\setminus L_{1,p+1})L_{2,q+1}$, hence
$\rad(Y)\geqslant L_{1,p+1}$. If $M=M_4$, then 
$X=(P \setminus L_{1,p+1})(Q \setminus L_{2,q+1})$, so $\rad(X) \ge L_{1,p+1}L_{2,q+1}$. 
If $M=M_6$, then $Y=Q^\#~\cup~L_{1,1}^\#L_{2,q+1}~\cup~L_{1,p+1}^\#(Q\setminus L_{2,q+1})$,
hence $\rad(Y \cup \{e\})\geqslant L_{2,q+1}$.

Let $M=M_2$. Then one obtains $Y=A \cup B \cup C$, where
\begin{equation}\label{eq:Ynew}
A=(P \setminus (L_{1,1} \cup L_{1,p+1}))L_{2,q+1},~
B=L_{1,p+1}(Q \setminus (L_{2,1} \cup L_{2,q+1}))~\text{and}~
C=L_{1,p+1}^\#L_{2,q+1}^\#.
\end{equation}
Write $(\underline{Y})^2= \sum_{g \in G} c_g  g$.
A straightforward computation shows that $c_h=p(q-2)^2+2(p-1)(q-2)$
for $h \in Q \setminus (L_{2,1} \cup L_{2,q+1})$, and
$c_{h^\prime}=(p-2)^2q+2(p-2)(q-1)$
for $h \in P \setminus (L_{1,1} \cup L_{1,p+1})$.
Since $Y$ is a basic set of $\A$ and
$h,\, h^\prime \in Y$, $c_h=c_{h^\prime}$ must hold. This gives that
$(q-p)(pq-2)=0$, a contradiction.

Finally, let $M=M_5$. Then
$Y=A \cup B \cup C \cup (L_{1,1}L_{2,1})^\#$, where $A,\, B$ and $C$ are defined 
in \eqref{eq:Ynew}.
In this case write $(\underline{Y})^2= \sum_{g \in G} \tilde{c}_g  g$.
Comparing the coefficients with the previous case ($M=M_2$) we see that
$\tilde{c}_h=c_h+2(q-2)$
for $h \in Q \setminus (L_{2,1} \cup L_{2,q+1})$, and
$\tilde{c}_{h^\prime}=c_{h^\prime}+2(p-2)$
for $h \in P \setminus (L_{1,1} \cup L_{1,p+1})$.
As $\tilde{c}_h=\tilde{c}_{h^\prime}$ must hold, we obtain $(q-p)pq=0$,
a contradiction again.
\end{proof}

The particular case when $|G|=36$ is considered separately and the proposition below
follows from the database of S-rings over groups of small order due to
Reichard~\cite{R}.

\begin{prop}\label{36}
Suppose that $G \cong C_2^2 \times C_3^2$ and
$\A$ is a primitive rational schurian S-ring over $G$ of rank at least $3$.
Then $\A$ has a basic set in the form
$H_1^\# \cup H_2^\#$,
where $H_1$ and $H_2$ are subgroups of $G$ of order $6$ and
$H_1 \cap H_2=\{e\}$.
\end{prop}

The subgroups $H_i$'s above and as well as in Theorem~\ref{primrat}
can be used to define a translation net with
translation group $G\r$ (see \cite[Definition~1]{BJ}) and this connection will be explored
in the proof of Theorem~\ref{main2}.
The rest of the section is devoted to translation nets.
\medskip

An \emph{$(n,k)$-net} $\N=(\Omega,\L)$ consists of
a set $\Omega$ of $n^2$ points and a set $\L$ of $kn$ lines
such that  
\begin{enumerate}[{\rm (1)}]
\item each line $L \in \L$ contains $n$ points,
\item $\L$ is partitioned into $k$ \emph{parallel classes}:
$\L_1,\ldots,\L_k$,
\item any two lines from distinct parallel classes intersect at
exactly one point.
\end{enumerate}
The \emph{collinearity graph} $\Gamma_\N$ has vertex set
$\Omega$, and two points $\alpha$ and $\beta$ are adjacent
if and only if there is a line $L \in \L$ passing through these points.
$\Gamma_\N$ is a strongly regular graph with
parameters $(n^2, k(n-1), n-2+(k-1)(k-2),k(k-1))$ and
non-principal eigenvalues $n-k$ and $-k$.

Following~\cite{BJ}, a \emph{weak automorphism} of $\N$ is a permutation of $\Omega$, which preserves the line set $\L$.
By a \emph{strong automorphism} we mean a weak automorphism
when, in addition, it also preserves each parallel class $\L_i$.
If $\N$ admits a group $H$ of strong automorphisms acting regularly
on $\Omega$, then it is called a \emph{translation net}
with \emph{translation group} $H$.

One way to construct an $(n,k)$-net is the following.
Let $H$ be a group of order $n^2$. A
\emph{partial congruence partition} of $H$ with degree $k$
(an \emph{$(n,k)$-PCP} for short) is a family of $k$ subgroups
$H_1,\ldots,H_k$ of order $n$ such that $H_i \cap H_j=\{e\}$ whenever
$i \neq j$. Letting $\Omega=H$ and $\L$ to be the set of
all right cosets $H_i x$, $1 \leqslant i \leqslant k$ and $x \in H$,
the pair $(\Omega,\L)$ becomes an $(n,k)$-net whose $i$-th
parallel class $\L_i$ consists of the cosets $H_i x,\, x\in  H$.
Note that, the collinearity graph is $\Gamma_{(\Omega,\L)}=\cay(H,D)$, where $D=\bigcup_{i=1}^k H_i^\#$. Furthermore, $(\Omega,\L)$ is a translation net
with translation group $H\r$.

\begin{prop}\label{net1}
Let $\N$ be an $(n,k)$-net such that $n > (k-1)^2$.
Then the size of a clique of the collinearity graph $\Gamma_\N$
is bounded by $n$, and the lines of $\L$ are the only
$n$-cliques of $\Gamma_\N$.
\end{prop}
\begin{proof}
Let $\Delta$ be a clique not contained in any line. We will show that
$|\Delta|\leqslant (k-1)^2$. Choose a line $L$ such that $m=|\Delta\cap L|$ is
as large as possible. As $\Delta\not\subseteq L$, we can take a point $\delta\in\Delta\setminus L$.
Let $L_0$ be the line through $\delta$ parallel to $L$, and let $L_1,\dots,L_m$ be the lines
connecting $\delta$ to the points in $\Delta\cap L$. Then $L_0,L_1,\dots,L_m$ are pairwise
distinct, so $m+1\leqslant k$. Each point in $\Delta$ is connected to $\delta$, hence we obtain
$|\Delta|\leqslant 1+k(m-1)\leqslant (k-1)^2$.
\end{proof}

\begin{prop}\label{net2}
Let $\N=(\Omega,\L)$ be an $(n,k)$-net such that $k< n$ and
$H$ be an abelian group of weak automorphisms of $\N$,
which is regular on $\Omega$. Then every element in $H$ is
a strong automorphism.
\end{prop}
\begin{proof} Note, first, that $|H|=|\Omega|=n^2$.
Let $L \in \L$ and $\O=\{L^h: h \in H\}$ be the
orbit of $L$ under $H$ in its action on $\L$.
It follows from  $|\O| \leq |\L|=nk < n^2 = |H|$ that the setwise stabilizer
$H_{\{L\}}=\{h\in H : L^h = L\}$ is non-trivial.
Since $H$ is abelian, $H_{\{L\}}=H_{\{L^h\}}$ for every $h \in H$.
In particular, the intersection $L\cap L^h$ is mapped to itself by
$H_{\{L\}}$. Using also that $H$ is semiregular on $\Omega$, this
implies that $|L\cap L^h|$ is divisible by $|H_{\{L\}}|$.
If $L \ne L^h$, then $|L\cap L^h| \in \{0,1\}$, and we conclude that
the lines in $\O$ are pairwise disjoint. Therefore,
$|H/H_{\{L\}}|=|\O| \leqslant n$, implying that  $|H_{\{L\}}| \geq n$.
On the other hand, $|H_{\{L\}}| \leqslant |L|=n$. Therefore,
$|H_{\{L\}}|=n=|\O|$ and, consequently, $\O$ is a parallel class of $\L$.
\end{proof}

We conclude the section with a sufficient condition for the CI-property of
an S-ring over $C_p^2 \times C_q^2$.

\begin{lem}\label{PCP}
Let $G$ be an abelian group of order $p^2q^2$ for primes $p<q$,
$A\in \Sup^{\min}(G\r)$ and $\A=V(G,A_e)$.
Suppose that there exists an $\A$-set of the form
$$
H_1^\# \cup \cdots \cup H_k^\#,
$$
where $H_1,\ldots,H_k$ are subgroups of $G$
and form a $(pq,k)$-PCP of $G$.
Then $\underline{H_i} \in \A$ for each $1 \leqslant i \leqslant k$.
If $k > 1$, then $\A$ is CI.
\end{lem}
\begin{proof}
Denote by $\N$ the induced translation net, i.e., the point set is $G$ and the lines are the cosets $H_ix$, $1 \leqslant i \leqslant  k$ and $x \in G$.
The collinearity graph is $\Gamma_\N = \cay(G,X)$, where
$X=\bigcup_{i=1}^{k}H_i^\#$.
Let $\gamma \in \aut(\A)$. Then $\gamma \in \aut(\Gamma_\N)$.
Now $k \leqslant p+1$, and we have
$(k-1)^2 \leqslant p^2 < pq$.
By Lemma~\ref{net1}, the lines of $\N$ are the only $n$-cliques of
$\Gamma_\N$. Since $\gamma \in \aut(\Gamma_\N)$, it follows that $\gamma$
maps an $n$-clique to an $n$-clique, and we conclude that
$\gamma$ is a weak automorphism of $\N$.

Let $F$ be a regular and abelian subgroup of $A$.
By Lemma~\ref{net2}, $F$ is a group of
strong automorphisms of $\N$, or equivalently,
the partition of $G$ into its $H_i$-cosets is $F$-invariant.
Thus $\underline{H_i} \in \A$ follows from the $\preceq_{G}$-minimality of $A$.
If $k >1$, then Lemma~\ref{ci-complementary} shows that $\A$ is CI.
\end{proof}
\section{Proof of Theorem~\ref{main2}}\label{sec:proof2}
For this section we fix the following notation:
\begin{quote}
$G = P\times Q$, where $P \cong C_p^2$, $Q \cong C_q^2$ for distinct primes $p$ and $q$,
$A \in \Sup^{\min}(G\r)$ and $\A=V(G,A_e)$.
\end{quote}

Again, our goal is to show that $\A$ is CI (see Corollary~\ref{ci}). In the proof we
shall use the following two lemmas.

\begin{lem}\label{X}
Suppose that $\underline{P} \notin \A$ and let $x \in P^\#$ be such that $\underline{\langle x \rangle}
\notin \A$.
Then for every basic set $X \in \cS(\A)$
\begin{enumerate}[{\rm (i)}]
\item $q$ divides $|X \cap Qx|$.
\item If $X$ contains an element of order $p$, then
$X \cap Qx$ can be one of the following: $\emptyset$, $Rx$ or $(Q \setminus R)x$ for some subgroup
$R \leqslant Q$ of order $q$, or $Qx$.
\end{enumerate}
\end{lem}
\begin{proof}
For (i) assume on the contrary that $|X \cap Qx|$ is not divisible by $q$.
Consider the $\A$-set $X^{[q]}$ defined in
Proposition~\ref{SW}(ii). Then $x^q \in X^{[q]}$, hence
$\langle X^{[q]} \rangle=\langle x \rangle$ or $P$,
contradicting that none of these subgroups are
$\A$-subgroups.

If $X$ also contains an element of order $p$, then it is $\P_q(G)$-invariant and
(ii) follows from this and (i) in the same way as in the proof of Lemma~\ref{trivial}.
\end{proof}

\begin{lem}\label{X-new}
Suppose that $\underline{P} \notin \A$, $\underline{Q} \in \A$, and there exists a subgroup
$U \le G$ of order $pq^2$ such that $\underline{U} \in \A$.
Let $X \in \cS(\A)$ be a basic set with the following properties:
there exist an element $x \in P \setminus U$ and a subgroup $R \leqslant Q$ of order $q$
such that $X \subseteq Ux$, $X \cap Qx=Rx$, $X \ne Rx$, $X \ne (U \cap P)Rx$. Then $\A$ is CI.
\end{lem}
\begin{proof}
Due to Proposition~\ref{KM2} and \eqref{eq:p2-p}, $\A_{G/Q} \cong \Z C_p^2$ or
$\Z C_p\wr \Z C_p$.
Since $X  \ne Rx$, the former case is impossible, hence $|X \cap Qxy^i|=q$ for every $0\leqslant i \leqslant p-1$, where $y$ is a  generator of $U_p=U \cap P$.
Notice that $\underline{\langle x  \rangle} \notin \A$ by Lemma \ref{ci-complementary}, and hence
Lemma~\ref{X}(ii) can be applied to $X$.
Thus for each $0\leqslant i \leqslant p-1$,
$$
X \cap Qxy^i=R_ixy^i~\text{or}~(Q \setminus R_i)xy^i,~R_i \leqslant Q~
\text{and}~|R_i|=q.
$$
The subgroup $R_0=R$ and if
$X \cap Qxy^i=(Q \setminus R_i)xy^i$ for some $i$, then $q=2$ must hold.
\medskip

\noindent{\bf Case~1.}
For each $0 \leqslant i \leqslant p-1$, $X \cap Qxy^i=R_ixy^i$.
\medskip

 Notice that, $X^{[p,k]} \cup \{e\}$ is just the union
of those subgroups $R_i$ that occur exactly $k$ times in the union
$$
X=R_0x \cup R_1xy \cup \cdots \cup R_{p-1}xy^{p-1}.
$$
Since $X \ne (U_p)Rx$, it follows that there exists an integer
$1 \leqslant k \leqslant p-1$ such that $X^{[p,k]}$ is non-empty $\mathcal{A}$-set,
see Proposition~\ref{SW}(ii). Hence $X^{[p,k]} \cup \{e\}=\bigcup_{i=1}^{r}S_r$ with
$\{S_1,\dots,S_r\} \subseteq \{R_0,\dots,R_{p-1}\}$.
By Proposition~\ref{SW}(ii), this is an $\A$-set.
Write $\underline{X} \cdot \underline{X^{[p,k]}}=\sum_{g\in G} c_g g$
and fix a generator $u_i$ of $R_i$ for each $0 \leqslant i \leqslant p-1$.
A direct computation shows that $c_{u_ixy^i}=q-1$ or $0$ depending on whether
$R_i \in \{S_1,\ldots,S_r\}$ or not. On the other hand,
all of the coefficients $c_{u_ixy^i}$ must be the same, and therefore,
$\{S_1,\ldots,S_r\}=\{R_0,\ldots,R_{p-1}\}$.
This means each $R_i$ occurs $k$ times, in particular, $k$ divides $p$.
It follows from $1 \le k \le p-1$ that
$k=1$, hence the subgroups $R_0,\dots,R_{p-1}$ are pairwise distinct.
Let $H_i=R_i \langle xy^i\rangle$,
It is easy to see that the subgroups $H_i$, $0 \leqslant i \leqslant p-1$ form
a $(pq,p)$-PCP of $G$. Since
$\trX \cup X^{[p]}=H_0^\# \cup \cdots \cup H_p^{\#}$ is an $\mathcal{A}$-set,
Lemma~\ref{PCP} gives that $\A$ is CI.
\medskip

\noindent{\bf Case~2.}
$q=2$ and $X \cap Qxy^i=(Q   \setminus R_i)xy^i$ for some
$1 \leqslant i \leqslant p-1$.
\medskip

We show that this case cannot occur.
Assume that $p > 3$.
Applying Lemma~\ref{sylow} for $L=\{e\}$ and $U$ yields
$\underline{U_p} \in \A$.
Let $l=|X \cap U_px|$. As $x\in X$ but $xy^i\notin X$, we have $1\leqslant l<p$.
Now the coefficient of $x$ in $\underline{X}\underline{U_p}$ is equal to $l$, hence
for any $u\in Q$ for which $X \cap (U_p)x u$ is not empty, we get $|X \cap U_px u|=l$.
Using also that $X \ne U_pRx$,
it follows from this that $|X|=3l$ or $4l$ since $|Q|=4$. This contradicts that $|X|=2p$ and $p > 3$.

If $p=3$, i.e., $|G|=36$, then the database of S-rings over groups of small order given  in~\cite{R} shows that $\A$ does not exist.
\end{proof}

We focus first on the case when $\A$ is decomposable.
\subsection{$\A$ is decomposable}
Let $\A$ be a non-trivial $S=U/L$-wreath product.
Since $\{e\} < L \leqslant U<G$, we have
$1 \leq \Omega(|L|)\leq \Omega(|U|)\leq 3$. If $\Omega(|U|)=\Omega(|L|)$, then
$U=L$ and $\A$ is CI by Corollary~\ref{ci-wp}.
So it remains to consider the following cases:
$$
(\Omega(|U|),\Omega(|L|)) \in \{ (3,2), (2,1), (3,1) \}.
$$

\noindent{\bf Case 1.} $(\Omega(|U|),\Omega(|L|))=(3,2)$.
\medskip

In this case $\Omega(|S|)=1$, hence $\A_S$ is Cayley minimal by Lemma~\ref{mix}(i) and (iii). We may assume w.l.o.g.~that
$$
|G/L|=pq~\text{or}~|G/L|=p^2.
$$

Let $|G/L|=pq$. As $\underline{S} \in \A_{G/L}$, $\rk(\A_{G/L}) \ne 2$.
According to Lemma~\ref{mix}(ii) $\A_{G/L}$ is cyclotomic or a non-trivial
wreath product of two S-rings.  We claim that $\A$ is CI in both cases.
Indeed, this follows by Proposition~\ref{ci-caymin} if $\A_{G/L}$ is cyclotomic.
If $\A_{G/L}$ is a non-trivial wreath product, then
$\A_{G/L}=\A_S \wr \A_{(G/L)/S}$. This implies that $\A=\A_U \wr \A_{G/U}$,
and so $\A$ is CI by Corollary~\ref{ci-wp}.

Now, suppose that $|G/L|=p^2$. By Proposition~\ref{KM2},  $\A_{G/L}$ is a $p$-S-ring,
in particular, $\A_S=\Z S$, and so $\A$ is CI by Corollary~\ref{ci-wp}.
\medskip

\noindent{\bf Case 2.} $(\Omega(|U|),\Omega(|L|))=(2,1)$.
\medskip

Again, $\Omega(|S|)=1$ and $\A_S$ is Cayley minimal.
We may assume w.l.o.g~that $|L|=p$. Then
$$
|U|=pq~\text{or}~|U|=p^2.
$$

Let $|U|=pq$. As $\underline{L} \in \A$, $\rk(\A_U) \ne 2$.
It follows from Lemma~\ref{mix}(ii) that $\A_U$ is cyclotomic or a non-trivial
wreath product of two S-rings. In the former case $\A$ is CI by
Proposition~\ref{ci-caymin}.
In the latter case $\A_U=\A_L\wr \A_{U/L}$, which implies that
$\A=\A_L \wr \A_{G/L}$, and so $\A$ is CI by Corollary~\ref{ci-wp}.

Now, suppose that $|U|=p^2$, i.e., $U=P$. Note that, $S$ is an $\A_{G/L}$-subgroup of
order~$p$. Denote the maximal $q$-$\A_{G/L}$-subgroup of $G/L$ by $H$.
Clearly, $|H| \in \{1,q,q^2\}$.

If $|H|=q^2$, then $\underline{LQ} \in \A$. By
Proposition~\ref{KM2}, $\A_{G/(LQ)}\cong \Z C_p$. This implies that
$\A_S=\Z S$, and so $\A$ is CI by Corollary~\ref{ci-wp}.

Let $|H| \in \{1,q\}$. By Proposition~\ref{MS1}, $\A_{G/L}$ is a
non-trivial $(HS)/S$-wreath product. This implies that $\A$ is the
$(HS)^{(\pi_{G/L})^{-1}}/U$-wreath product.
One can see that $(\Omega(|(HS)^{(\pi_{G/L})^{-1}}|),$ $\Omega(|U|))=(2,2)$ or
$(3,2)$, and hence we are done by Corollary~\ref{ci-wp} or by Case~1, respectively.
\medskip

\noindent{\bf Case 3.} $(\Omega(|U|),\Omega(|L|))=(3,1)$.
\medskip

In this case $\Omega(|S|)=2$. We may assume w.l.o.g.~that $|U|=p^2q$.
By Proposition~\ref{KM2}, $\A_{G/U} \cong \Z C_q$, and
this implies that
\begin{equation}\label{eq:GLS}
(\A_{G/L})_{(G/L)/S} \cong \Z C_q.
\end{equation}

Clearly, $|L| \in \{p,q\}$. Let $|L|=q$.
Then $|S|=p^2$.  Using this, Eq.~\eqref{eq:GLS} and Proposition~\ref{MS2}, we
find that $\A_{G/L}=\A_S \star \A_{Q_1}$, where $Q_1$ is the least
$\A_{G/L}$-subgroup of $G/L$ of order divisible by $q$.
If $|S\cap Q_1|=1$, then $\A_{G/L}=\A_S \otimes \A_{Q_1}$.
Then $\A$ is CI by Lemma~\ref{ci-gwp}.
If $|S\cap Q_1|>1$ then $\A_{G/L}$ is the non-trivial $S/(S\cap Q_1)$-wreath product.
Thus $\A$ is the $U/(S\cap Q_1)^{(\pi_{G/L})^{-1}}$-wreath product.
One can see that
$(\Omega(|U|),\Omega(|(S\cap Q_1)^{(\pi_{G/L})^{-1}}|))=(3,3)$ or $(3,2)$, and hence we are done by Corollary~\ref{ci-wp} or by Case~1.

Now, suppose that $|L|=p$.
Then $|G/L|=pq^2$ and $|S|=pq$.
Denote by $H$ the unique maximal $q$-$\A_{G/L}$-subgroup of $G/L$, and
by $P_1$ the least $\A_{G/L}$-subgroup of order divisible by $p$.
Obviously, $|H|\in\{1,q,q^2\}$.

Let $|H| \in \{1,q\}$. Assume that $H \nleqslant S$. Then $|H|=q$ and
$|H\cap S|=1$. So $G/L=H\times S$ and Eq.~\eqref{eq:GLS} implies that
$\A_H \cong  \Z C_q$. Then $\A_{G/L}=\A_S \otimes \A_H$ by Proposition~\ref{EKP},
and hence $\A$ is CI by Lemma~\ref{ci-gwp}.
Now, let $H \leqslant S$.  Since $|S|=pq$, $P_1 \leqslant S$.
By Proposition~\ref{MS1},  $\A_{G/L}$ is the $S/P_1$-wreath product.
Thus $\A$ is the $U/P_1^{(\pi_{G/L})^{-1}}$-wreath product. We are done by 
Corollary~\ref{ci-wp} or Case~1
because $(\Omega(|U|),\Omega(|P_1^{(\pi_{G/L})^{-1}}|))=(3,3)$ or $(3,2)$.

Let $|H|=q^2$, $V=LQ$ and $K=H \cap S$.
Then $\underline{V} \in \A$ and
$\underline{K} \in \A_{G/L}$.
By Proposition~\ref{KM2}, $\A_{G/V} \cong \Z C_p$, implying that,
\begin{equation}\label{eq:GLH}
(\A_{G/L})_{(G/L)/H} \cong \Z C_p.
\end{equation}

Assume for the moment that $H$ contains an $\A_H$-subgroup $K^\prime$ of order
$q$ such that $K^\prime \ne K$.
It follows from Eq.~\eqref{eq:GLS} that $(\A_{G/L})_{H/K} \cong \Z C_q$,
this implies that $\A_{K^\prime}=\Z K^\prime$.
Then $\A_{G/L}=\A_S \otimes \A_{K^\prime}$ by
Proposition~\ref{EKP}, and so $\A$ is CI by Lemma~\ref{ci-gwp}.

From now on $K$ is assumed to be the only $\A$-subgroup of $H$ of order
$q$. The S-ring $\A_{G/L}=\A_H \star \A_{P_1}$  by Proposition~\ref{MS2} and
Eq.~\eqref{eq:GLH}.

Let $|H \cap P_1|=1$ then $\underline{P} \in \A$, and so
$\A_{G/P}$ is a $q$-S-ring by Proposition~\ref{KM2}.
This implies that $\A_K=\Z K$. On the other hand, $\A_ {P_1}=\Z P_1$
follows from Eq.~\eqref{eq:GLH}, and we conclude that $\A_S=\A_{P_1K}=\Z S$,
and so that $\A$ is CI by Corollary~\ref{ci-wp}.

Now, suppose that $|H\cap P_1|>1$. Then $P_1=S$ and it follows that
$\A_{G/L}$ is the $H/K$-wreath product.
By Proposition~\ref{DW}, $\A_H=\A_K \wr \A_{H/K}$, and
we conclude that $\A_{G/L}=\A_K \wr \A_{(G/L)/K}$.
Thus $\A$ is the $U/K^{(\pi_{G/L})^{-1}}$-wreath product, and we are done by Case~1 because
$(\Omega(|U|),\Omega(|K^{(\pi_{G/L})^{-1}}|))=(3,2)$.
By this we have considered all cases and shown that $\A$ is CI.
\subsection{$\A$ is indecomposable}
Assume first that $\A$ is primitive. If $\rk(\A)=2$, then $\A$ is CI, hence let
$\rk(\A) > 2$. If $p$ or $q$ is equal to $2$ and $|G| > 36$, then by
Lemma~\ref{sylow} (choose $L=\{e\}$ and $U=G$),
$\underline{P} \in \A$ if $q=2$ and $\underline{Q} \in \A$ if
$p=2$. Hence either $p, q > 2$ or $|G|=36$.  Since $\rk(\A) > 2$, $\rk(\trA) > 2$
as well, see Lemma~\ref{rank2}.  Also, $\trA$ is primitive and as
$\trA=\A \cap W(G)$, it is also schurian.
Due to Theorem~\ref{primrat} and Proposition~\ref{36},
there exists a basic set of $\trA$ of the form
$H_1^\# \cup \cdots \cup H_k^\#$, where $H_1,\ldots,H_k$ are subgroups of $G$ and
form an $(pq,k)$-PCP of $G$. By Lemma~\ref{PCP}, each $H_i$ is an $\A$-subgroup,
a contradiction.
\medskip

Now let $\A$ be imprimitive and $U$ be a proper $\A$-subgroup of maximal order.
\medskip

\noindent{\bf Claim.} $|G/U|$ is a prime.
\medskip

If $|G/U|$ is a prime power, then $\A_{G/U}$ is a primitive $p$-S-ring, implying that
$|G/U|$ is a prime.
If $|G/U|=pq$, then by Lemma~\ref{sylow}, $UR$ is an $\A$-subgroup, where
$R \in \{P,Q\}$ is a Sylow $\max(p,q)$-subgroup of $G$. This contradicts the
maximality of $U$.

Thus we have to consider a unique case: $U$ is of prime order. W.l.o.g. $|U|=q$.
Suppose there exists a proper $\A$-subgroup $V \ne U$. Then
$U \cap V=1$ and $UV=G$, and hence $|V|=|G|/|U|>|U|$, contrary to the maximal choice of $U$. Thus we assume that $U$ is the unique non-trivial proper $\A$-subgroup.
The quotient $\A_{G/U}$ is a primitive S-ring over an abelian group of order $p^2q$.
By Proposition~\ref{W} it has rank $2$.
Therefore $TU=G \setminus U$ holds for each basic set $T$ outside
$G \setminus U$. It follows from Lemma~\ref{centre} that for each $u \in U$ the singleton $\{u\}$ is a basic set of $\A$. Therefore $Tu$ is a basic set of $\A$.
Thus either $U \leqslant \rad(T)$ or $T \cap Tu = \emptyset$ for each $u \in U^\#$.
In the first case $\A$ is the wreath product $\A_U \wr \A_{G/U}$, contradicting that
$\A$ is indecomposable.
In the second case, $T,\, Tu,\, \ldots,\, Tu^{q-1}$ are the only basic sets outside $U$.
Therefore $|T| = p^2q-1$.

If $T \cap Q = \emptyset$, then $Tu \cap  Q = \emptyset$ for all $u \in U$. Therefore, every basic set $Tu$ contains $q$-elements.
It follows from $T \cup Tu \cup \cdots \cup Tu^{q-1}=G \setminus U$ that at least one
of the sets $Tu^i$ contains some $p$-element. W.l.o.g.~$T \cap P \neq \emptyset$. Combining this with $T \cap Q \neq \emptyset$ we conclude that $T$ is rational.
Write
$$
T=T_0 \cup P_1^\# S_1 \cup \cdots \cup P_{p+1}^\# S_{p+1},
$$
where $P_1,\ldots,P_{p+1}$ are all subgroups of $P$ of order $p$,
$T_0:=T \cap Q$ and $S_i \subseteq Q$, $1 \leqslant i \leqslant p+1$ are rational sets.
It follows from
$TU = G \setminus U$ that each $S_i$ is a transversal of $Q/U$.
Therefore $|S_i|=q$ and $S_i$ is a subgroup of
$Q$ of order $q$ such that $S_i \neq U$. Now, it follows from
$|T| = |T_0|+(p+1)(p-1)q$ that
$|T_0| = q-1$. Since $T_0$ is rational, the subset $S_0:=T_0 \cup \{e\}$ is a subgroup of $Q$ distinct from $U$. Note that the subgroups $S_i < Q$, $0 \leqslant i \leqslant p+1$ are not necessarily distinct.  Let us choose the indices so that
$S_0, \ldots, S_k$ is the complete set of distinct subgroups that appear in the list
$S_0, \ldots, S_{p+1}$. Denote by $n_i$ the multiplicity with which $S_i$,
$0 \leqslant i \leqslant k$
appears in the above list. Then $\sum_{i=0}^{k}n_i = p+2$.

We finish the proof of the claim by finding a non-trivial proper $\A$-subgroup distinct
from $U$, which was excluded above.
If $0 < i \leqslant k$ and $n_i \ne p$, then the set $S_i^\#$ is contained in $T^{[p]}$. Therefore, if $k \ne 0$ and $(k,n_0,n_1) \ne (1,2,p)$,
then $T^{[p]}$ is a non-empty subset of $Q$, which intersects $U$ trivially, and so
$\langle T^{[p]} \rangle$ is the required $\A$-subgroup.
If $k=0$, then $T^{[p]}=\{e\}$ and $T=(PS_0)^\#$,
showing that $PS_0$ is an $\A$-subgroup of order $p^2q$.
Finally, if $(k,n_0,n_1)=(1,2,p)$, then $T^{[p]}=\{e\}$ and
$T=(P_jS_0)^\# \cup (P\setminus P_j)S_1$ for some $1 < j \leqslant p+1$.
Then $P_j \le \rad(T\cup \{e\}) < G$, hence $\rad(T\cup \{e\})$ is the required
$\A$-subgroup. The claim is proved.
\medskip

Assume w.l.o.g.~that $|U|=pq^2$. For the rest of the proof
fix an element $x \in P \setminus U$.
By Proposition~\ref{KM2}, $\A_{G/U} \cong \Z C_p$, and this
shows that $Ux$ is an $\A$-set. Let
$I$ be the intersection of all subgroups $\rad(X)$, $X \in \cS(\A)$ and
$X \subseteq Ux$.
Then $\underline{I} \in \A$. Furthermore, it follows from Proposition~\ref{SW}(i) that $I \leqslant \rad(X)$ for any basic set $X$ outside $U$, and we find that $\A$ is the $U/I$-wreath product. As $\A$ is indecomposable, $I=\{e\}$. We have shown the following:
\begin{equation}\label{eq:rad}
\bigcap_{X \in \cS(\A) \atop X \subseteq Ux} \rad(X)=\{e\}.
\end{equation}

Notice that $\underline{\langle x^\prime \rangle} \notin \A$ can also be assumed
for each $x^\prime \in P \setminus U$, otherwise $\A$ would be CI by
Lemma~\ref{ci-complementary}.
\medskip

\noindent{\bf Case~1.} $\underline{P} \notin \A$.
\medskip

Let $y$ be a generator of $U_p=U \cap P$. For $0 \leqslant i
\leqslant p-1$, let
$X_i$ be the basic set containing $xy^i$. By Lemma~\ref{X}(ii),
$$
X_i \cap Qxy^i=R_ixy^i
$$
for some non-trivial\ subgroup $R_i \leqslant Q$.
Note that the sets $X_i$ are not necessarily distinct.
In view of Eq.~\eqref{eq:rad}, we may assume w.l.o.g.~that $|R_0|=q$.

Fix $0 \leqslant i \leqslant p-1$. We claim that every basic set
$X \in \cS(\A)$ satisfies
\begin{equation}\label{eq:Ri}
X \cap Qxy^i \ne \emptyset,~
\underline{R_i} \in \A~\text{and}~|R_i|=q \implies R_i \leqslant \rad(X).
\end{equation}
Indeed, $\{u\} \in \cS(\A)$ for every $u \in R_i$ because of
Lemma~\ref{centre}, hence the right multiplications $\rho_G(u)$, $u \in R_i$
map the basic sets of $\A$ having non-empty intersection with $Qxy^i$ to themselves. Lemma~\ref{X}(i) implies that there are at most $q$ such basic sets.
Using this and that $X_iR_i=X_i$, we conclude that $XR_i=X$, i.e., \eqref{eq:Ri} holds.

Assume first that $X_0=U_pR_0x$. Then $X_i=X_0$ and $R_i=R_0$ for each 
$1 \leqslant i \leqslant p-1$.
It follows from Eq.~\eqref{eq:rad} that
$U_p \nleqslant \rad(X)$ for some basic set $X \subset Ux$.
Then $S:=\langle X^{[p]} \rangle$ is a non-trivial $\A$-subgroup contained in
$Q$ and distinct from $R_0$. If $|S|=q$, then
$\{u\} \in \cS(\A)$ for every $u \in S$ because of Lemma~\ref{centre}, and we
find that the sets $X_0 u=U_pR_0xu$, $u \in S$ are the basic sets contained in $Ux$.
This contradicts Eq.~\eqref{eq:rad}.
Let $S=Q$. Then $R_0=\rad(X_0) \cap S$, and so $\underline{R_0} \in \A$.
Then \eqref{eq:Ri} implies that $R_0 \leqslant \rad(X)$ for every basic set 
$X \subset Ux$, a contradiction to Eq.~\eqref{eq:rad}.

Assume second that $X_0 \ne U_pR_0x$ and $\underline{Q} \in \A$.
By Lemma~\ref{X-new}, we may assume that
$X_0=R_0x$. Then $\underline{R_0} \in \A$ and $\A_{G/Q} \cong \Z C_p^2$ by Proposition~\ref{KM2}. Thus for every $0 \leqslant i \leqslant p-1$, $X_i=R_ixy^i$,
and hence $\underline{R_i} \in \A$. It follows from Eq.~\eqref{eq:rad} and the 
implication in \eqref{eq:Ri} that
there exists some $1 \leqslant i \leqslant p-1$ such that $|R_i|=q$ and
$R_i \ne R_0$. But then $\langle X_0 \rangle=\langle R_0x \rangle$ and
$\langle X_i \rangle=\langle R_ixy^i \rangle$
are $\A$-subgroups intersecting trivially, and hence
$\A$ is CI by Lemma~\ref{ci-complementary}.

We are left with the case when $X_0 \ne U_pR_0x$ and $\underline{Q} \notin \A$. We show that $\underline{R_0} \in \A$.
Assume on the contrary that $\underline{R_0} \notin \A$.
As neither $\underline{Q}$ belongs to $\A$,
$X^{[p]}$ contains no element from $R_0^\#$, and hence
$R_0^\#xy^i \subseteq X \cap Qxy^i$ if $1 \leqslant i \leqslant p-1$.
Using also Lemma~\ref{X}, we find that the latter intersection is equal to $R_0xy^i$, $Qxy^i$ or
$(Q \setminus R)xy^i$ for some subgroup $R < Q$ such that $|R|=q$ and $R \ne R_0$.
Furthermore, as $X \ne U_pR_0x$, there exists some $1 \leqslant i \leqslant p-1$
such that one of the last two possibilities holds.
If $X \cap Qxy^i=Qxy^i$, then $(Q \setminus R_0) \subseteq X^{[p]}$, hence
$\langle X^{[p]} \rangle=Q$, a contradiction.
Suppose that $X \cap Qxy^i=(Q \setminus R)x$.
Then $Q \setminus (R \cup R_0) \subseteq X^{[p]}$, which implies that $q=2$.
If $p>3$, then Lemma~\ref{sylow} applied to $\A$, where $|G/\{e\}|=4p^2$,  
we find that $\underline{P} \in \A$, contradicting our assumption
$\underline{P} \notin \A$. If $p=3$,
then $|G|=36$, and it follows from the database of S-rings of small order in \cite{R} that
such an $\A$ cannot exist.

Note that, the above proof also shows that, for any $1 \leqslant i \leqslant p-1$,
$\underline{R_i} \in \A$ whenever
$|R_i|=q$. Now, since $\underline{Q} \notin \A$, $R_0$ is the only non-trivial $\A$-subgroup
contained in $Q$, and therefore, $R_i=R_0$ or $Q$ for each
$1 \leqslant i \leqslant p-1$. Then by \eqref{eq:Ri}, $R_0 \leqslant \rad(X)$ for each 
basic set $X \subseteq Ux$, a contradiction to Eq.~\eqref{eq:rad}.

\medskip

\noindent{\bf Case~2.} $\underline{P} \in \A$.
\medskip

By Proposition~\ref{KM2}, there exists an $\A$-subgroup $V$ of order $p^2q$.
If now $\underline{Q} \notin \A$, then one can copy the argument
used in Case~1 with $V$ and $Q$ playing the role of
$U$ and $P$, respectively and deduce that $\A$ is CI. If $\underline{Q} \in \A$, then  Lemma~\ref{ci-complementary} shows that $\A$ is CI.

\end{document}